\documentclass[amscd]{amsart}

\usepackage{amsmath}
\usepackage{amssymb}
\usepackage{amsfonts}
\usepackage{a4}

\theoremstyle{definition}
\newtheorem{theorem}{Theorem}[section]
\newtheorem{lem}[theorem]{Lemma}
\newtheorem{prop}[theorem]{Proposition}
\newtheorem{cor}[theorem]{Corollary}
\newtheorem{definition}[theorem]{Definition}

\newtheorem{claim}[theorem]{Claim}
\newtheorem{example}[theorem]{Example}

\newtheorem{conjecture}[theorem]{Conjecture}
\newtheorem{remark}[theorem]{Remark}
\newtheorem{note}[theorem]{Note}

\newtheorem{Thm}[theorem]{Theorem}
\newtheorem{Claim}[theorem]{Claim}

\newtheorem{Prop}[theorem]{Proposition}
\newtheorem{Lem}[theorem]{Lemma}
\newtheorem{Def}[theorem]{Definition}
\newtheorem{Rem}[theorem]{Remark}

\newtheorem{Cor}[theorem]{Corollary}
\newtheorem{Conj}[theorem]{Conjecture}

\newcommand{\mC}{{\mathbb C}}

\newcommand{\mS}{\mathbb S}

\newcommand{\mZ}{{\mathbb Z}}

\newcommand{\ho}{\hookrightarrow}
\newcommand{\Gg}{\gamma}

\newcommand{\bO}{\Omega}

\newcommand{\bs}{\sigma}

\newcommand{\ep}{\epsilon}
\newcommand{\D}{\Delta}

\newcommand{\kk}{\kappa}

\newcommand{\mcB}{\mathcal B}
 
\newcommand{\mcD}{\mathcal D}
\newcommand{\mcE}{\mathcal E}

\newcommand{\mcG}{\mathcal G}

\newcommand{\mcK}{\mathcal K}

\newcommand{\mcO}{\mathcal O}

\newcommand{\mcS}{\mathcal S}

\newcommand{\fg}{\mathfrak g}

\newcommand{\fm}{\mathfrak m}

\newcommand{\fz}{\mathfrak z}
\newcommand{\ti}{\tilde}
\newcommand{\wt}{\widetilde}

\newenvironment{pf}{\proof[\proofname]}{\endproof}

\newcommand{\CC}{{\mathcal C}}
\newcommand{\PP}{{\mathcal P}}

\renewcommand{\H}{{\mathcal H}}

\newcommand{\Hb}{{\overline{\mathcal H}}}

\newcommand{\Ht}{{\tilde \H}}

\newcommand{\Endtil}{{\widetilde {End}}}

\newcommand{\Hh}{\hat{\mathcal H}}

\newcommand{\fB}{{\mathfrak B}}
\newcommand{\fC}{\mathfrak C}

\newcommand{\Pone}{{\mathbb P}^1}
\newcommand{\Aone}{{\mathbb A}^1}

\newcommand{\Lotimes}{\overset{\rm L}{\otimes}}

\newcommand{\uHom}{\underline{\mathrm{Hom} }}

\newcommand{\F}{{\mathcal F}}
\newcommand{\A}{{\mathcal A}}
\newcommand{\MM}{{\mathcal M}}
\renewcommand{\O}{{\mathcal O}}

\newcommand{\Z}{{\mathfrak Z}}
\newcommand{\Zbar}{{\overline \Z}}
\newcommand{\Zb}{{\overline \Z}}
\newcommand{\Zt}{{\tilde \Z}}
\newcommand{\Ztil}{{\tilde \Z}}

\newcommand{\g}{{\mathfrak g}}

\newcommand{\Ve}{{\mathcal V}}
\newcommand{\KK}{{\mathcal K}}

\newcommand{\Zet}{{\mathbb Z}}
\newcommand{\cupl}{\bigcup\limits}

\newcommand{\suml}{\sum\limits}
\newcommand{\imbed}{\hookrightarrow}
\newcommand{\iso}{{\widetilde \longrightarrow}}

\newcommand{\To}{\longrightarrow}
\newcommand{\Qu}{{\mathbb Q}}
\newcommand{\RE}{{\mathbb R}}
\newcommand{\Ce}{{\mathbb C}}

\newcommand{\rank}{\operatorname{rank}}

\newcommand{\Smb}{{\overline{Sm}}}
\newcommand{\Rb}{{\overline{R}}}

\newcommand{\bA}{\overline{\A}}
\newcommand{\bInd}{\overline{\text{Ind}}}
\newcommand{\uInd}{\underline{\text{Ind}}}

\newcommand{\LT}{{^LT}}

\newcommand{\al}{\alpha}
\newcommand{\wh}{\widehat}

\newcommand{\la}{\lambda}
\newcommand{\La}{\Lambda}

\newcommand{\mcA}{\mathcal K}

\begin{document}

\title[Character values and Hochschild homology]{Character values and Hochschild homology}
\author{Roman Bezrukavnikov, David Kazhdan}

\begin{abstract}
We present a conjecture (and a proof for $G=SL(2)$) generalizing a result of J.~Arthur
which expresses a character value of a cuspidal representation of a $p$-adic
group as a weighted orbital integral of its matrix coefficient. It also generalizes 
a conjecture by the second author proved  by Schneider-Stuhler and (independently) the first author.
The latter statement expresses an elliptic character value as an orbital integral of a
 pseudo-matrix coefficient defined via the Chern character map taking
value in zeroth Hochschild homology of the Hecke algebra. 
The present conjecture generalizes the construction of pseudo-matrix coefficient
using  compactly supported Hochschild homology, as well as a modification
of the category of smooth representations, the so called compactified category
of smooth $G$-modules.
This newly defined "compactified pseudo-matrix coefficient" lies in a certain
space $\KK$ on which the weighted orbital integral is a conjugation invariant linear functional,
our conjecture states that evaluating a weighted orbital integral on the compactified pseudo-matrix
coefficient one recovers the corresponding character value of the representation. 

We also discuss the properties of the averaging map from $\KK$ to the space of invariant distributions, partly building
on  works of Waldspurger and Beuzart-Plessis. \end{abstract}

\maketitle

\tableofcontents

\tableofcontents

\section{Introduction}

Let $G$ be a reductive 
group over a local nonArchimedean field $F$.

The goal of the article is to present an algebraic expression for a character of an admissible representation of $G$ on a compact element.

The statement is presented as a conjecture (see Conjecture \ref{char_form_conj})
for a general reductive group, it is proved in the paper for $G=SL(2)$. We also describe a modification
of the category $Sm=Sm(G)$ of finitely generated smooth representations, the so called {\em compactified category} of smooth $G$-modules,
 which plays a key role in our algebraic description of character values and may have an independent interest.

To describe the context for these constructions recall a conjecture of \cite{cusp}
proved in \cite{SS} and \cite{thes}.

Let $\H=\cup_K \H_K$ be the Hecke algebra of locally
constant compactly supported $\Ce$-valued measures. Thus $Sm(G)$ is identified with the category
of finitely generated nondegenerate $\H$ modules \cite{BZ}.

Let $C(\H)=\H/[\H,\H]=\H_G=HH_0(\H)=HH_0(Sm)$ be the cocenter
of $\H$. Here $\H_G$ denotes coinvariants with respect to the conjugation action,
while  $HH_*$ stands for Hochschild homology, and its
second appearance refers to the notion of Hochschild homology of an abelian category.

Since $\H$ is Noetherian and has finite homological dimension, there is a well defined 
Chern character (also called the Hattori-Stallings or Dennis trace) map $ch: K^0(Sm)\to C(\H)$
(we will abbreviate $ch([M])$ to $ch(M)$). It has been conjectured in \cite{cusp}
and proven in \cite{SS}, \cite{thes}
that for an {\em elliptic} regular semisimple element $g\in G$ and an admissible
representation $\rho$ we have
\begin{equation}\label{ell}
\chi_\rho(g)=O_g(ch(\rho)),
\end{equation}
where $O_g$ denotes the orbital integral. Here we  use that $O_g:\H\to \Ce$ being
 conjugation invariant factors through $C(\H)$.
 
 If $\rho$ is a cuspidal irreducible representation then (assuming that $G$ has a compact center) a matrix coefficient
 $m_\rho\in \H$  is a representative of the class $ch(\rho)\in C(\H)$. 
 Thus in this case \eqref{ell} reduces to an earlier result of Arthur \cite{Ar1}.
 However, the latter applies also to nonelliptic regular semisimple
 elements: for such an element $g$ and a cuspidal irreducible
 representation $\rho$ Arthur has proved that 
 \begin{equation}\label{Ar_f}
 \chi_\rho(g)=WO_g(m_\rho),
 \end{equation}
 where $WO_g$ denotes the {\em weighted orbital integral}. 
 Our Conjecture \ref{char_form_conj} provides a generalization of  \eqref{ell} to all
regular semisimple compact elements $g$, which for a cuspidal representation $\rho$ reduces to \eqref{Ar_f}.

The first step  in this direction is a generalization of the map $ch:K^0(Sm)\to C(\H)$.
For our present purposes we need to modify both the source and the target
of this map. We replace the target $C(\H)=\H_G$  by  $\KK_G$ where 
 $\KK\subset \H$ is a subspace invariant under
the conjugation action of $G$, the so called space of 
"weightless"
functions.\footnote{This \label{ftnt} adjective reflects the fact that weighted orbital integrals
restricted to this space are independent of the choices involved in choosing the weight function
on an orbit.
This space has appeared in the literature (see \cite{BP}, \cite{W} 
and references therein) where it was called the space of {\em strongly} cuspidal function. We refrain from using this terminology since
we use the term "cuspidal function" in the sense of \cite{cusp} where it refers to a function
acting by zero in any parabolically induced representation, thus in our terminology
$\KK$ contains the space of cuspidal functions. The term "cuspidal
function" was used in a different sense in
\cite{BP}, \cite{W} etc., so that $\KK$ is contained in the set of cuspidal functions in the sense of {\em loc. cit.} }
 Definition and some properties of $\KK$ are discussed in
section \ref{sec_K}. The key property is that $WO_g|_\KK$ is a $G$-invariant
functional for any regular semisimple element $g\in G$; furthermore, there is a well defined averaging map $Av$
from $\KK$ to the space of invariant generalized functions on $G$ and for $\phi\in \KK$
the value of $WO_g(\phi)$ coincides with $Av(\phi)(g)$ (the latter is well defined
since $Av(\phi)$ is in fact a locally constant function on the set of regular elements
in $G$). We also provide a conjecture with a proof for $PGL(2,F)$, $char(F)=0$ 
describing the image of the averaging map from $\KK$ to 
the space of invariant distributions. 

To describe the source of the map generalizing $ch$
 we need some new ingredients. One of them is
the so called {\em compactified category} of smooth (finitely generated) representations $\Smb$.
 
 The abelian category $\Smb$ is defined in section \ref{compcat}.
 Recall that according to Bernstein \cite{centre}, $Sm$ can be identified
 with the category of coherent sheaves of modules over a certain
 sheaf of algebras over a scheme which is an infinite union of affine
 algebraic varieties, the spectrum $\Z$ of the Bernstein center of $G$. The category $\Smb$ can be described as the
 category of coherent sheaves of modules over a certain coherent sheaf of algebras over a
 (componentwise) compactification of $\Z$.
 An admissible module $\rho$ can also be viewed as an object in
 $\Smb$, so we can apply the Chern character to the class of $\rho$
 obtaining  $\bar{ch}(\rho)\in HH_0(\Smb)$.
 
 We also need another invariant of the category $Sm$, namely
 the {\em compactly supported} Hochschild homology $HH_*^c(Sm)$
 which is the derived global sections with compact support in the 
 sense of \cite{D} of localized Hochschild homology
 $R\underline{Hom}_{\H\otimes \H^{op}}(\H,\H)$.
 
 We have natural maps  $HH_*^c(Sm)\to HH_*(\Smb)\to HH_*(Sm)$;
 for an admissible module $\rho$ we have its compactly supported
 Chern character $ch^c(\rho)\in HH_0^c(Sm)$, so that $\bar{ch}(\rho)$
 and $ch(\rho)$ equal the images of $ch^c(\rho)$ under the corresponding
 maps.
 
The first statement in the main conjecture (a theorem for $SL(2)$) provides a natural isomorphism 
$$\KK^c_G \cong Im(HH_0^c(Sm)\to HH_0(\Smb)),$$ 
where $\KK^c\subset \KK$ is the subspace of measures supported on compact elements.
By the previous paragraph, $\bar{ch}(\rho)$ belongs to that image, thus
we obtain a homological construction of an element in $\KK^c_G$ 
from an admissible representation, the so called "compactified pseudo-matrix coefficient"
of the representation. 

The second main statement  (proved for $SL(2)$) asserts that
for a compact regular element $g$ and an admissible representation $\rho$
 we have $WO_g(\bar{ch}(\rho))=\chi_\rho(g)$. Notice that for a noncompact
 regular element $g$ the value of $\chi_\rho(g)$ coincides with a character
 value of the  Jacquet functor applied to $\rho$ \cite{DCas},
 in particular it vanishes for a cuspidal module.

We view Conjecture \ref{char_form_conj} as an  algebraic statement underlying some aspects of Arthur's local trace
formula \cite{Ar}, while equality \eqref{ell} underlies the elliptic part of the local trace formula (see also
Remark \ref{int_rem} below).
We plan to develop this theme in a future publication.

\bigskip
 
{\bf{Acknowledgements.} }
We thank Joseph Bernstein, Vladimir Drinfeld, Dmitry Kaledin and Dmitry Vaintrob for many useful discussions
over the years. In particular, the definition of the compactified category was conceived as a result of discussions
with Kaledin and it took its present form partly due to discussions with Drinfeld.

We also thank 
Rapha\" el Beuzart-Plessis, Dan Ciubotaru, Eric Opdam and Jean-Loup Waldspurger for helpful correspondence.

The project received funding from ERC under grant agreement No 669655;
R.B. was partly supported by  NSF grant DMS-1601953, 
the collaboration was partly supported by  the US-Israel Binational Science Foundation.

\section{Weightless functions and invariant distributions}\label{sec_K}
Let $\H=\H(G)$ be the Hecke algebra of compactly supported locally
 constant measures on $G$, the convolution product on $\H(G)$ will be denoted by $*$. 
 For an open subsemigroup $S\subset G$ we let $\H(S)\subset \H(G)$ denote the subalgebra
 of measures supported on $S$.
 We denote by $\mcD$ the space of generalized functions on $G$ 
 (that is the space of linear functionals on $\H (G)$) and by $\mcD ^G\subset \mcD $  the subspace of invariant generalized functions.
 Let $\H_{cusp}$, $\mcD^G_{cusp}$ be the cuspidal part of $\H$, $\mcD^G$,
 i.e. $\H_{cusp}$ consists of functions acting by zero in any parabolically
 induced representation 
  and $\mcD^G_{cusp}$ consists of distributions vanishing
 on the orthogonal complement of $\H_{cusp}$ (cf. footnote \ref{ftnt} above).
Until the end of the section we assume for simplicity of notation that $G$ has compact center.
 
 Then  {\em averaging}  with respect to  conjugations yields a well defined map $H_0(G,\H_{cusp})\to \mcD^G_{cusp} $
 given by $f\mapsto \int\limits_{G} \ \frac{{^gf}}{dg} dg$ for a Haar measure $dg$ on $G$.

In this section we define a larger subspace $\mcA $ in $\H$
on which the averaging map is still well defined and 
 conjecture that the map $\tau$ defines an embedding 
$H_0(G,\mcA )\ho \mcD ^G$. Moreover, we propose a conjectural
 description of the  image of $\tau$. We prove this conjecture in
 the case when $G=PGL(2)$. 

\subsection{The conjecture}
We start with some notation. Let 
$ \mcO \subset F$ be the ring of integers  and $\pi $ be a
 generator of  the maximal ideal $\fm$ of $\mcO$. Let 
 $val :F^\star \to \mZ$ be the valuation such that $val(\pi)=1$. 
We define $\| x\| =q^{-val (x)}, x\in F^\star$ where $q=|\mcO / \fm|$.

For any smooth $F$-variety $X$ we denote by $\mcS (X)$ the space
 of locally constant measures on $X$ with compact support. In the
 case when $X$ is a homogeneous $G$-variety with a $G$-invariant
 measure $dx$ the map $f \to f /dx$ identifies $\mcS (X)$ with
the space of compactly supported locally constant functions
 on $X$. In this case we will not distinguish between functions
 and measures on $X$. Also, we will work with spaces $Y_P=(G/U_P\times G/U_P)/L$
 for a parabolic subgroup $P=LU_P\subset G$. In this case
 $\mcS(Y_P)$ will denote the space of integral kernels of operators
 $\mcS(G/U_P)\to \mcS(G/U_P)$, i.e. locally constant compactly supported
 sections of the $G$-equivariant locally constant sheaf $pr_1^*(\mu)$, where
 $pr_1:Y_P\to G/P$ is the first projection.

In particular, we fix a Haar measure $dg$
 on $G$ and identify the space  $\mcS (G)=\H(G)$ with 
 the space of compactly supported locally constant functions
 on $G$. Then $\D$ is identified with the space of distributions,  we also get the $L^2$ pairing $\langle\ ,\ \rangle$ on $\H(G)$.


Let $\fg$ be the Lie algebra of $G$. We will assume that $\fg$
has a finite number of  nilpotent  conjugacy classes and that there exists a $G$-equivariant $F$-analytic  bijection $\phi$ between  a neighbourhood of $0$ in $\fg$ and  a neighbourhood of $e$ in $G$. 
We also assume that for a semisimple element $s\in G$  the Lie algebra 
$\fz$ of its centralizer $Z_G(s)$ admits a  $Z_G(s)$-invariant complement. 
These assumptions are well known to hold if $char(F)=0$ or if $char(F)> N$ for some $N$ depending on the rank of $G$,
see \cite[\S 1.8]{KV} for more precise information.

\begin{definition}
\label{distributions}

\begin{itemize}
\item We denote by $\mcG _{G ,e}$ the space of germs of $Ad$-invariant distributions near $0$ on $\fg$ which are restrictions
 of  linear combinations  of the  Fourier transforms of invariant
 measures on a nilpotent orbits. Using the bijection $\phi$ we 
 consider  $\mcG _{G,e}$ as  a space of germs of $Ad$-invariant
 distributions on $G$ near the identity.

\item For a semisimple element $s\in G$ we denote by
 $ \mcG _{Z_G(s),s}$ the   space of germs of distributions  on
 $Z_G(s)$ at $s$ obtained from the space $\mcG _{Z_G(s),e}$ by
 the shift by $s$.
 
 \item Let $s\in G$ be a semisimple element, $ X_s=G/Z_G(s), r:G\to  X_s$ be the natural projection and  $\Gg :X_s\to G$ be a continuous section.
We denote by $dz$ a $G$-invariant measure on $X_s$.

\item 
We denote by $\ti \kk :Z_G(s)\times X_s \to G$  the map given by
$(z,x)\to \Gg (x)sz(\Gg (x))^{-1}$. 

\end{itemize}
\end{definition}

Let $\fz\subset \fg$ be the Lie algebra of $Z_G(s)$. By assumption there exists a complementary $\fz$-invariant subspace 
$W\subset \fg$ and the map $\kk _0:\fz \oplus W\to \fg ,(z,w)\to z+Ad(s)w$ is a bijection. Therefore there exists an open neighborhood $R\subset Z_G(s)$ of $e$
such that the restriction $\kk$ of $\ti \kk$ on $R\times X_s$ is an open embedding.

\begin{definition}

\label{distribution0}

For any $ \bar \psi \in  \mcG _{Z_G(s),s}$ we choose a representative 
 $\ti \psi \in \mcD ^{Z_G(s)}(Z_G(s))$ of $\bar \psi$.

\begin{itemize}
\item For a function $f\in \mcS (G)$ and $z\in R\subset X_s$ we define 
$f_z \in \mcS (Z_G(s))$ by $f_ z:=f(\kk (z,x))$.

\item We define a function $\bar f$ on $R$ by $\bar f(z):=\ti \psi (f_z)$.

\item We define a distribution 
$\psi (\ti \psi) $ on $G$ by $\psi (f):=\int _{X_s}\bar f(z)$.

\item We denote by $[\psi (\ti \psi)]$ the germ of the distribution $\psi$ at $s$. 

\end{itemize}

\end{definition}

It is clear that for any two choices $\ti \psi , \ti \psi '$
of representatives of $\bar \psi$ the difference $\psi (\ti \psi) -\psi (\ti \psi ')$ vanishes on a $G$-invariant open neighborhood of $s.$ Therefore the germ $[\bar \psi]$
does not depend on a choice of a representative of $\bar \psi$.
 

\begin{definition}
\label{distributions'}

\begin{itemize}
\item We denote by $\mcG _s$ the space of germs at $s$  of 
$Ad$-invariant distributions of the form 
$[\psi ], \psi \in \mcG _{Z_G(s),s}$.

\item We denote by  $\mcE \subset  \mcD ^G$ the subspace of
 distributions $\alpha$ such that 

a) there exists a compact subset $C$ in $G$ such that $supp
 (\alpha)\subset G(C)$ and 

b) for any semisimple $s\in G $ the germ of $\alpha$ at $s$
 belongs to $\mcG _s$.

\end{itemize}

\end{definition}

\begin{remark}\label{Ealt}
If $char(F)=0$ then $\mcE$ admits an equivalent description as the space
of invariant distributions $\alpha$ satisfying the following requirements:

a) there exists a compact subset $C$ in $G$ such that $supp
 (\alpha)\subset G(C)$;

b) there exists a compact open subgroup $K\subset G$ 
such that for every element $z$ in the Bernstein center 
satisfying  $\delta_{K}*z=0$  we have
$\alpha*z=0$. 

Equivalence of the two definitions of $\mcE$ follows 
from\footnote{We thank Rapha\" el Beuzart-Plessis for pointing this out to us.}
\cite[Theorem 16.2]{HC}.

\end{remark}

\begin{definition}
\begin{itemize}

\item 
We define the space  $ \mcA (G)$ of {\it weightless}
 functions as the subspace in $\mcS (G)$ of functions $f$ such
 that  $\int _{u\in U_Q} f(lu)du=0, l\in L$
for all proper parabolic subgoups $Q=LU_Q\subset G$.

\item For a closed conjugation invariant subset $X$ of $G$ 
we define the space  $ \mcA (X)=\mcA_X\subset \mcS (X)$ 
 as the subspace of functions $f$ such that 
 $$\int _{u\in U_Q} f(lu)du=0, $$
for all proper parabolic subgroups $Q=LU_Q\subset G$ and $l\in L$
 such that $lU_Q\subset X$.
\end{itemize}
\end{definition}

\begin{remark}\label{diag_rem}
For $f\in \H$ and a parabolic $P=LU_P\subset G$ let $A_P(f)\in \mcS(Y_P)$ denote its orishperic
transform, i.e. the integral kernel of the action of $f$ on $\mcS(G/U_P)$. Let $\Delta_{Y_P}\subset Y_P$
be the preimage of diagonal under the projection $Y_P\to (G/P)^2$. Then $\Delta_{Y_P}\cong (G/U_P\times L)/L$,
where $L$ acts on the first factor by right translations and on the second one by conjugation.
It is easy to see that for $f\in \H$ we have $f\in \KK$ iff for any parabolic subgroup $P\subsetneq G$ we have $A_P(f)|_{\Delta_{Y_P}}=0$. 
\end{remark}

The following result is due to J.-L. Waldspurger
(see  \cite[Lemma 9]{W}).

\begin{prop}
\label{Waldspurger}
For $f\in \H(G)$ the following are equivalent:

a) $f\in \KK$.

b) For any $h\in \H(G)$ the function $g\mapsto \langle ^gf, h\rangle$
has compact support.

\end{prop}

For any $f\in \mcS (G)$ we define  distribution
$\hat f $ by:
$$\langle \hat f ,h\rangle :=\int _{g\in  G}\langle ^gf, h  \rangle dg.$$



For future reference we mention the following. 

\begin{Lem} For $f\in \KK$ the distribution $\hat{f}|_{G^{rs}}$ 
(where $G^{rs}$ is the open set of regular semisimple elements)
is a locally constant function. For $g\in G^{rs}$ we have
$$\hat f (g) =WO_g(f),$$
where $WO_g$ denotes the {\em weighted orbital integral}.
\end{Lem}

{\em Proof\ } follows from the definition and basic properties of the weighted orbital integral, see e.g. \cite[\S I.11]{ArIn}. \qed  

The group $G$ acts on $\mcA$ by conjugation.
It is clear that the map 
$f\to \hat f$
factors through a map 
$\tau : H_0(G,\mcA )\to  \mcD ^G$. For any $f\in \mcA$ we
 denote by $[f]$ it image in $H_0(G,\mcA )$.

\begin{conjecture}
\label{imb} a) $\hat{f}\in \mcE$ for $f\in \mcA$.

b) The map $\tau$ defines an isomorphism between $H_0(G,\mcA )$ and $\mcE$.

c) $\dim (H_0(G,\mcA (\bO) ))=1$ for any regular semisimple conjugacy class
 $\bO \subset G$. 
\end{conjecture}
\begin{remark}
One can check that $\hat{f}$ satisfies the conditions of Remark \ref{Ealt}, thus
if $char(F)=0$ then part (a) of the conjecture follows from Harish-Chandra's Theorem
\cite[Theorem 16.2]{HC}; see 
\cite[Corollary 5.9]{W0}, 
\cite[Proposition 5.6.1]{BP} for details.
\end{remark}
\begin{remark} It is clear that part $c)$ follows from $a)$ and $b)$.
\end{remark}

\begin{remark}\label{c_ell}
If the centralizer of an element $g\in \bO$ is an anisotropic (compact) torus then
statement (c) clearly follows from uniqueness (up to scaling) of a Haar measure on $G$.
In the case when that centralizer has split rank one  the statement is checked in the next subsection.
\end{remark}

\subsection{Almost elliptic orbits} To simplify the wording we assume in this subsection that the center of $G$ is compact.
A regular semisimple element $g\in G$ will be called almost elliptic if the split rank of its centralizer
is at most one. We now prove Conjecture \ref{imb}(c) in the case when $\bO$ consists of almost elliptic
elements. 

Fix $g\in \bO$ and let $T$ be the centralizer of $G$, thus $\bO\cong G/T$. In view of Remark \ref{c_ell} it suffices to consider the case when the split 
 rank of $T$ equals one; we also assume without loss of generality that $G$ is almost simple. 
 
 There are exactly two parabolic subgroups $P,\, P'\subsetneq G$ containing $T$. Let $U$, $U'$
 be their unipotent radicals. 
 
 Consider the complex
 \begin{equation}\label{compl_ae}
 0\to \mcA_\bO\to \mcS(G/T)\to \mcS(G/TU)\oplus \mcS(G/TU')\to \Ce\to 0;
 \end{equation}
 here $\mcS$ stands for the space of locally constant compactly supported measures as before, 
 the third arrow send $\phi$ to $(pr_*(\phi), pr'_*(\phi))$, where $pr:G/T\to G/TU$, $pr':G/T\to G/TU'$
 are the projections and the fourth arrow sends $(\phi, \phi')$ to $\int \phi-\int \phi'$.
 \begin{lem}\label{compl_ae_lem}
 The complex \eqref{compl_ae} is exact. 
 \end{lem}
\proof Exactness at all the terms except for  $\mcS(G/TU)\oplus \mcS(G/TU')$ is clear.
Suppose that $(\psi,\psi'):\mcS(G/TU)\oplus \mcS(G/TU')\to \Ce$ is a linear functional 
vanishing on the image of $\mcS(G/T)$. Let $\tilde\psi:\mcS(G)\to \Ce$ be the composition 
of the direct image map $\mcS(G)\to \mcS(G/TU)$ with $\psi$. Then $\psi$ is a right $TU$ invariant
generalized function on $G$. On the other hand, $-\psi$ is equal to the composition 
of the direct image map $\mcS(G)\to \mcS(G/TU')$ with $\psi'$, which shows that $\tilde \psi$
is also right $TU'$ invariant. Since $G$ is assumed to be almost simple,  $U$ and $U'$ together generate $G$, thus we see that $\tilde \psi$
is right $G$ invariant. It follows that  $\tilde \psi$ is proportional to the functional $\phi\mapsto \int \phi$,
hence the functional $(\psi,\psi')$ factors through the differential in \eqref{compl_ae}, which yields
exactness of \eqref{compl_ae}. \qed

We can now finish the proof of Conjecture  \ref{imb}(c) in the present case. 
 Breaking \eqref{compl_ae} into short exact sequences we get
 $$ 0\to \mcA_\bO\to \mcS(G/T)\to M\to 0,$$
$$ 0\to M\to \mcS(G/TU)\oplus \mcS(G/TU')\to \Ce\to 0.$$
 Considering the corresponding long exact sequences on homology we see that
 it suffices to check that the map $\Ce=H_0(G, \mcS(G/T))\to H_0(G,M)$ is nonzero
 while the map $H_1(G, \mcS(G/T)) \to H_1(G,M)$ has one-dimensional cokernel.
 The former statement is clear since the composition $H_0(G, \mcS(G/T))\to H_0(G,M)\to
 H_0(G,\mcS(G/TU))$ is nonzero. To check the latter recall that $H_i(G,\Ce)=0$ for $i>0$
 since the resolution of $\Ce$ provided by the simplicial complex for computation of homology
 of the Bruhat-Tits building $\fB$ shows that $H_i(G,\Ce)\cong H_i(\fB/G)$, while $\fB/G$ is
 a product of simplices. Thus 
 $$H_1(G,M)\cong H_1(G,  \mcS(G/TU)\oplus \mcS(G/TU'))\cong H_1(T,\Ce)\oplus H_1(T,\Ce).$$
 Since $\mcS(G/TU)\cong H_1(T,\Ce)$ we see that 
 $$CoKer\left( H_1(G, \mcS(G/T)) \to H_1(G,M) \right)\cong H_1(T,\Ce),$$
 which is one-dimensional. \qed
 
\subsection{The case of $PGL(2)$}
To simplify the argument we assume in this subsection that $char(F)=0$. 


\begin{theorem} 
\label{imbedding}Conjecture \ref{imb} is true for $G=PGL(2)$.

\end{theorem}


The rest of the subsection is devoted to the proof of the Theorem.




\begin{claim} We have $\mcA _{cusp}\subset \mcA$.
\end{claim}

\begin{prop}
\label{inj}The map $\tau : H_0(G,\mcA)\to \mcD ^G$ is an embedding.
\end{prop}
\begin{pf} We need more notation.

\begin{definition}
\begin{itemize}
\item For $\ep \geq 0$ we define 
$G^ \ep =\{ g\in G| |p (g)| \leq \ep\}$, where $p(g)=\frac{tr^2(\ti g)}{det (\ti
 g)}-4$; here $\ti g\in GL(2,F)$ is a representative of $g$.

\item We let $G_s$, $G_e$, $N$, $\bar N$ denote, respectively, the sets
of regular semisimple split, regular semisimple elliptic, regular unipotent
and all unipotent elements. 

\item We set  $$\mcA _u=\mcA (\bar N):=\{ f\in \mcS (\bar N)\  | \int \limits_{u\in U}f 
(u)du =0\ \forall B=TU\subset G\},$$
$$\mcA _0:=\{ f\in \mcA _u| f(e)=0\},$$
where $B$ runs over the set of Borel subgroups in $G$.

\item For  $f\in \mcA$ we denote by $\kk (f)\in \mcA _u$ the restriction of $f$ to $\bar N$ and by $[\kk (f)]$ the image of $\kk (f)$ in $H_0(G, \mcA _u)$.

\end{itemize}
\end{definition}
We start with the following  geometric statement.
\begin{lem}
\label{extension} Let $f\in \mcS (G)$ be such that $f|_{\bar N}\in \mcA _u$. Then there exists $\ep >0$ such  that $f|_{G^\ep}\in \mcA $. 
\end{lem} 
\begin{pf}
Recall notations of Remark \ref{diag_rem}. We have $\Delta_{Y_B}=G/B\times T$, where $T$ is the (abstract) Cartan
subgroup of $G$.  It is easy to see that condition 
$f|_{\bar N}\in \mcA _u$ is equivalent to vanishing  of the restriction of $A_B(f)$ to $G/B\times \{1\}\subset \Delta_{Y_B}$.
Also,  condition  $f|_{G^\ep}\in \mcA $ is equivalent to vanishing of $A_B(f)$ on $G/B\times T_\ep\subset \Delta_{Y_B}$,
where $T_\ep=G_\ep\cap T$ (here we abuse notation by identifying the abstract Cartan subgroup $T$ with an arbitrarily 
chosen Cartan subgroup). Since $A_B(f)$ is locally constant for $f\in \H$, the statement follows from compactness of $G/B$.
\end{pf}

\begin{cor}
\label{extensions}  For any $f\in \mcA$ such that 
$[\kk (f)]=0$ there exists $f'\in \mcA$  with the same image in $H_0(G,\mcA)$ and such that $f'|_{\bar N}=0$.
\end{cor}
\begin{pf} Since $[f| _{\bar N}]=0$ we can write the restriction 
of $f$ to $\bar N$ as a finite sum $\sum _i (\ti f_i^{g_i}-\ti f_i), \ti f_i \in  \mcA _0', g_i\in G$. As follows from the  Lemma \ref{extension} we can choose  $f_i\in \mcA$ such that $\ti f_i={f_i}|_{\bar N}$. Then the function $f':=f-\sum _i ( f_i^{g_i}-f_i)$
satisfies the conditions of Corollary.
\end{pf}


\begin{prop}
\label{N'} 
The space $H_0(G,\mcA _u)$ is two dimensional.
\end{prop}
\begin{pf} 

We first show that $\dim (H_0(G,\mcA _0))=1$. 

Let $\mcB$ be the  variety of Borel subgroups, $p:N\to \mcB$ the map which associates to $u\in N$ the Borel subgroup containing $u$. By definition we have an exact sequence 
$$0\to \mcA _0\to \mcS (N)\to \mcS (\mcB )\to 0$$
and therefore an exact sequence 
$$H_1(G, \mcS (N) )\to H_1(G, \mcS (\mcB) )\to H_0 (G, \mcA _0)\to H_0(G ,\mcS (N))\to   H_0 (G ,\mcS (\mcB )).$$

\begin{lem} $H_1(G, \mcS (N) )=0 $.
\end{lem}
\begin{pf} Fix a Borel subgroup $B=TU$. 
We can write $U$ as a union of open compact subgroup 
$U_1\subset ...U_n\subset ...$. Therefore $\mcS (N)=\mcS (G/U)$
 is the direct limit of $\mcS (G/U_n)$. Since the functor $M\mapsto H_1(G,M)$
 commutes with  direct limits it is sufficient to show that  
 $H_1(G, \mcS (G/U_n) )=\{ 0\}$.  Since $U_n\subset G$ is a
 compact subgroup the space $H_1(G, \mcS (G/U_n) )$ is a direct
 summand of $\mcS (G)$. Since  $H_1(G, \mcS (G) )=0$ the
 Lemma is proven.
\end{pf}
Since $G$ acts transitively on $N$ and on $\mcB$ we have 
$H_0 (G ,\mcS (N))\iso H_0 (G ,\mcS (\mcB ))=\mC.$ 
Since $H_1(G, \mcS (N) )=0 $ we see that the map 
 $H_1(G, \mcS (\mcB) )\to H_0 (G, \mcA _0)$ is an isomorphism.
Since $\mcB =G/B$ we have:  
$$H_1(G, \mcS (\mcB) )=H_1(B,\mC)=H_1(T,\mC)=\mC.$$

 So $\dim( H_0 (G, \mcA _0))=1$. 
 
 To conclude the argument, recall the 
 short exact sequence $0\to \mcK _0\to \mcK _u\to \mC \to 0$.

Since $H_1(G,\mC)=0$ we have  an exact sequence: 

$$0\to H_0(G,\mcK _0) \to H_0(G,\mcK _u)\to \mC \to 0$$
So $\dim(H_0(G,\mcK _u)=2$. 
\end{pf}

Let $r:\mcK _{cusp}\to \mcK _u$ be the restriction and $[r]:\mcK _{cusp}\to  H_0(G,\mcK _u)$ be the composition of $r$ and  projection $\mcK _u\to H_0(G,\mcK _u) $.
\begin{lem}\label{r_onto}
The map $[r]$ is onto.
\end{lem}
\begin{proof} 
Recall the map 
$\tau: H_0(G,\mcK ) \to \mcD ^G$, $[f]\to \hat f$. Let $\bar \mcD ^G$ be the space of germs of invariant distributions at $e$ and 
$\bar{\tau}: H_0(G,\mcK ) \to \bar \mcD ^G$ be the composition of $\tau$ with the restriction map. 

Corollary \ref{extensions} implies that the map $\bar{\tau}$ vanishes on the kernel of the map $H_0(G,\mcK )\to
H_0(G,\mcK_u )$. Thus it suffices to show that $\bar{\tau}|_{\mcK _{cusp}}$ has rank at least two, i.e. that there
exist irreducible cuspidal representations $\rho_1$, $\rho_2$, such that their characters restricted to any
$G$-invariant open neighborhood of identity are not proportional. This is easily done by inspecting the 
character tables, see e.g. \cite[\S 2.6]{Sil}. 
\end{proof}

 \begin{cor}
\label{cusp}

a) For any $f\in \mcA$ there exists $f_{cusp} \in \mcA _{cusp}$ such that $[\kk (f)]=[\kk (f_{cusp'})]$.

b) For any $f_0\in \mcA _u$ there exists $f\in \mcA _{cusp}$ such that $[f_0]=[\kk (f)]$. \qed

\end{cor}

\bigskip

Let   $s\in G$ be  a regular split semisimple element and $\bO \subset G$ be the conjugacy class of $s$.

\begin{prop}
\label{regular} $\dim  (H_0(G, \mcA _\bO))=1$. 
\end{prop}
\begin{pf}

Let  $T=Z_G(t)$ be the split torus and 
$B,B'\subset G$  be Borel subgroups containing $T$. Since $\bO =G/T$ we have maps $r:\bO \to G/B$ and 
 $r':\bO \to G/B'$ and therefore morphisms 
 $r_\star :\mcS (\bO )\to \mcS (G/B)$ and
 $r'_\star :\mcS (\bO )\to \mcS (G/B')$.

As a special case of Lemma \ref{compl_ae_lem} we get:

\begin{lem} \label{SL2ex} The sequence 
\begin{equation}\label{star}
 0\to \mcA _\bO\to \mcS (\bO) \to \mcS (G/B) \oplus  \mcS (G/B') \to \mC \to 0,
 \end{equation}
where the last map $l$ is given by $(\nu ,\nu ')\mapsto \int \nu -\int \nu '$, is exact.  \qed
\end{lem}

Let $L:=ker (l)$. We have  an exact sequence
$$ 0 \to L \to \mcS (G/B) \oplus  \mcS (G/B') \to \mC \to 0.$$

Using that $G$ 
has homological dimension one, we get that the corresponding long exact sequence of homology contains the following fragment:
$$ 0   \to H_1(G,L)\to H_1(B, \mC )\oplus  H_1(B', \mC )\to  H_1(G, \mC ).$$
It is easy to  see that 
 $$\dim(H_1(B, \mC ) )=\dim (H_1(B',\mC))=\dim (H_1(T,\mC))=1$$
 Since the quotient $G/[G,G]$ is finite we see that $H_1(G, \mC )=0$. Therefore $\dim(H_1(G,L))=2$.

On the other hand we have  an exact sequence
$$0 \to  \mcA _\bO \to \mcS (\bO) \to L \to  0$$ 
and therefore  an exact sequence
$$H_1(G, \mcS  (\bO ))\to H_1(G, L )\to 
H_0 (G,  \mcA _\bO )\to H_0(G ,\mcS (\bO)).$$
Since $H_1(G, \mcS  (\bO ))=H_1(T, \mC )$ we see
 $\dim(H_1(G, \mcS (\bO))=1$.

\begin{lem}
The map $a:H_1(G, \mcS  (\bO ))\to H_1(G, L )$ is an embedding.
\end{lem}
\begin{pf} 
It is sufficient to show that the  map 
$H_1(G, \mcS  (\bO ))\to H_1(G, \mcS  (G/B))$ induced by the composition 
$p_\star\circ a:L\to \mcS (G/B)$
is an embedding. 

 Since $H_1(G, \mcS  (\bO ))=H_1(T, \mC )$, $H_1(G, \mcS  (G/B ))=H_1(B, \mC )$ and $H_{>0}(U,\mC)=0$,
  we see that this map is an isomorphism.
\end{pf}

We can now finish the proof of Proposition \ref{regular}. Since $G$ acts transitively on $G/B$ the map $\mu \to \int\limits _{G/B}\mu$ defines an isomorphism 
$H_0(G,\mcS (G/B))\to \mC$. On the other hand, since $\int \limits_{\bO }\nu =0$ for any $\nu \in \mcA _\bO$ the map 
$H_0 (G, \mcA _\bO)\to H_0(G ,\mcS (\bO))$ equals zero.
So $\dim  (H_0(G, \mcA _\bO))=1$. 
\end{pf}

Recall that $G_s\subset G$ is the subset of regular split semisimple
 elements, let $\mcA _s\subset \mcA$ be the subspace of functions in $\mcA$ 
 supported on $G_s$.
We fix a Cartan subgroup $T$, the Weyl group 
$W= \Zet/2\Zet$  acts on $T$ and on $G/T$ in the usual way.
 Then the map  $$(T-\{ e\}) \times G/T \to G_s,
 (s, g)\to gsg^{-1}$$  induces an isomorphism
\begin{equation}\label{Phit}
\mcS (G_s)\to (\mcS(T-\{e\})\otimes \mcS (G/T))^W.
\end{equation}

\begin{cor}
\label{split}
a) For $f\in \mcA_s$ the distribution $\hat f$ is locally constant
on $G_s$.

b) The map
 $f\mapsto (t\mapsto \hat f(t))$  induces an isomorphism 
$$
H_0(G,\mcA _s)\to   \mcS(T-\{e\})^W.$$ 

\end{cor}
\begin{pf} The isomorphism \eqref{Phit} is clearly compatible with the averaging map 
$f\mapsto \hat f$, which implies $a)$. Likewise, restriction to an orbit is compatible
with averaging, thus in view of \eqref{Phit} it suffices to show that for a fixed
orbit $\bO\subset G_s$ the map  $f\mapsto \hat f (t)$, $t\in \bO$ induces an isomorphism
$H_0(G,\mcA(\bO))\to \Ce$. By Proposition \ref{regular} it suffices to see that this map
is nonzero. This follows, for example, from the fact that the character of a
cuspidal representation does not necessarily vanish on $G_s$, while the character
of an irreducible cuspidal representation is obtained by averaging from its matrix coefficient. 
\end{pf}

Now we can finish the proof of Proposition \ref{inj}.
Let $f\in \mcA$ be such that $\hat f=0$. It follows from
 Corollary \ref{extensions} and Lemma \ref{extension} that we can
 can find $f'$ with the same image in $H_0(G,\mcA)$ such that   $f'=f_s+f_e$  
 where $f_s$ is supported on regular split semisiple elements and
 $f_e$ on regular elliptic elements. The condition $\hat f=0$
 implies that $\hat{f'}=0$, hence $\hat f_s=0$ and $\hat f_e=0$. It is easy to see
 that the condition   $\hat f_e=0$ implies that $[f_e]=0$. So we
 may assume that  the support of $f$ is contained in the subset
 $G_s\subset G$ of regular split semisimple elements. Now  Proposition \ref{inj} follows
 from Corollary \ref{split}.
\end{pf}

\bigskip

Let $\ti \mcD _e$ be the space of germs of  distributions at $e$ and  $ \mcD _e\subset \ti \mcD _e$ be the subspaces spanned by germs of characters of irreducible representations.

\begin{lem}
\label{germs} 
The space $\mcD _e$ is $2$-dimensional. It is spanned  by germs of characters of irreducible cuspidal representations. 
\end{lem}

\begin{pf} The second statement is a special case of a theorem of Harish-Chandra \cite{HC}. The first one also follows from {\em loc. cit.}, as it shown
there that more generally the space  $\mcD _e$ has a basis indexed by unipotent orbits.
\end{pf}


Recall that $\mcE \subset  \mcD$ is the subspace of distributions $\alpha$ satisfying the following three conditions: 

a) There exists a compact subset $C$ in $G$ such that $supp (\alpha)\subset C^G$.

b) The restriction of $\alpha$ on $G-\{ e\}$ is given by a locally constant function.

c) The germ of $\alpha$ at $e$ belongs to $\mcD _e$.

\begin{lem} $\tau (\mcA )\subset \mcE$.
\end{lem}
\begin{pf} Fix $f\in \mcA$. It is clear that the distribution $\hat f$ satisfies condition $a)$.

To prove that  $\hat f$ satisfies condition $b)$ we have to show that for any semisimple element $s\in G-\{ e\}$ there exists an open neighborhood $R\subset G$ of $s$ such that 
the restriction $\hat f|_{R}$ is a constant. If $s$ is split then this follows from Corollary \ref{split}(a), if $s$ is elliptic the proof is similar.


To prove that  $\hat f$ satisfies  condition $c)$ we observe
 that Corollary  \ref{cusp} implies  existence of 
 $f_{cusp}\in \mcA _{cusp}$ such that 
$[\kk (f)]=[\kk (f_{cusp})]$.  

It is easy to see that when $f$ is a matrix coefficient of an irreducible cuspidal 
representation $\rho$ then $\hat f$ is proportional to the character of $\rho$.
Thus condition $c)$ is satisfied by $\alpha_{cusp}=\hat{f_{cusp}}$. However, by 
Lemma \ref{extension} and Corollary \ref{extensions} the germs of $\alpha$ and $\alpha_{cusp}$ at $e$ coincide. 
\end{pf}


\begin{prop}
\label{sur}  $\tau (\mcA )= \mcE$.
\end{prop}

\begin{pf} 
It remains  to show that every  $\alpha \in \mcE$ is in the image of
 $\tau$.   Lemma \ref{germs} shows that
there exists $\beta \in \mcE$ which is a linear combination
of characters of cuspidal representations such $\alpha - \beta$
vanishes on an open neighborhood of $e$. Thus we have 
$\alpha-\beta=\alpha_s+\alpha_e$, where $\alpha_s$ is supported
on $G_s$, while $\alpha_e$ is supported on $G_e$.

Now $\alpha_s$ is in the image of $\tau$ by Corollary \ref{split},
while $\alpha_e$  is in the image of $\tau$ by a similar
argument. Also, $\beta$ is in the image of $\tau$ since
the character of an irreducible cuspidal representation 
$\rho$ equals $\hat{f}$ where $f$ is a matrix coefficient
of $\rho$.

Proposition \ref{sur} and therefore Theorem \ref{imbedding}
are proven.
\end{pf}

\section{The compactified category of smooth modules} 
\label{compcat}
 \subsection{Definition of the compactified category} \label{def_comp}
 For a parabolic $P=LU$  let $L^0\subset L$ be the
subgroup generated by compact subgroups; thus $L^0$ is the kernel of the
unramified characters of $L$. Set $\check{\Lambda}_P=L/L^0$.


Let $\Lambda_P$ be the group of $F$-rational characters of $L$ and $\Lambda_P^+$ be the subset
of $P$-dominant weights, i.e. weights which are (non-strictly) dominant with respect to any
(not necessarily $F$-rational) Borel subgroup $B\subset P$. 
We have a nondegenerate
pairing between the lattices $\check{\Lambda}_P$ and $\Lambda_P$ given by:
$$\langle x L^0, \la \rangle = val_F (\la (x)).$$


Let  $\check{\Lambda}_P^+$  be the subsemigroup
defined by:
$$ \check{\Lambda}_P^+ =\{ x\in \check{\Lambda}_P \ |\  \langle x, \la \rangle \geq 0 \ \ \forall \la\in \La_P^+\},$$
and let $L^+_P\subset L$
be the preimage of $\check{\Lambda}_P^+$ under the projection 
$L\to \check{\Lambda}_P$. 

For a pair of parabolics $P\supset Q$ let $L_Q^{P+}
\subset L_Q$ denote 
the image of $L_P^+\cap Q$ in $L_Q=Q/U_Q$.
It  is easy to see that  $L_Q^{P+}
\supset L_Q^+$. 

For an open submonoid $M\subset G$ we let $Sm(M)$ denote the category
of nondegenerate finitely generated $\H(M)$-modules; this is easily seen 
to be equivalent to the category of finitely generated smooth $M$-modules.

For parabolic subgroups $P\supset Q$ we have the "Jacquet" functor $J^P_Q: Sm (L_P^+)\to 
Sm(L_Q^{P+})$, $M\mapsto M_{\overline{U_Q}}$, where $\overline{U_Q}$ is image of $U_Q$ in
$L_P=P/U_P$. 

\medskip

To simplify the wording in the following definition we fix a minimal parabolic
$P_0$, then by a standard parabolic we mean a subgroup $P$ containing $P_0$.
 
\begin{Def}\label{sm_def}
 The {\em compactified category} of smooth $G$-modules  $\Smb^{qc}=\Smb^{qc}(G)$ 
  is the category whose
 object is a collection $(M_P)$ indexed
by standard 
parabolic subgroups $P=L_PU_P$, where $M_P$ is a smooth
module over $L_P^+$, together with isomorphisms
\begin{equation}\label{isom_struk}
J^P_Q(M_P)\cong \H( L_Q^{P+})\otimes _{\H(L_Q^+)} M_Q
\end{equation}
fixed for every pair of standard parabolic subgroups $P=L_PU_P\supset
Q=L_QU_Q$; here $\H$ denotes the algebra of locally constant compactly supported distributions. The isomorphisms 
are required to satisfy the associativity identity  for each triple of parabolics $P_1\supset P_2\supset P_3$.

An object in the compactified category is called coherent if the module $M_P$ is finitely generated for
all $P$.

We let $\Smb=\Smb(G)\subset \Smb^{qc}$ denote the full subcategory of coherent objects.
\end{Def}


It is easy to see that $\Smb(G)$ is an abelian category, the functor sending $(M_P)\in \Smb$ to $M_G$
identifies  $Sm(G)$ with  a Serre quotient of $\Smb$.

We also have an adjoint functor  $Sm(G)\to \Smb^{qc}$. This functor sends admissible modules
but not general finitely generated modules to $\Smb$.

\begin{example}\label{sl2ex}
Let $G=SL(2)$. In this case the category $\Smb$ admits the following more direct
description. A component of the spectrum $\Z$ of Bernstein center  is in this case
either a point or an affine curve, thus  $\Z$ admits a canonical (componentwise) compactification $\Zb$.
Notice that $\partial \Z=\Zb\setminus \Z$ is identified with the set $(O^\times)^*$ of characters of 
$T_0=O^\times$. Let $T^+=O\setminus \{0\}\subset F^\times =T$, 
thus $T^+\cong O^\times\times \Zet_{\geq 0}$. 
Set $\Zt^+=Spec(\H(T^+))\cong (O^\times)^*\times \Aone$. 
Notice that 
we have a natural map $\pi:\Zt^+ \to 
\Zb$ inducing an isomorphism $(O^\times)^*\times \{0\}\to \partial \Z$.
Moreover, $\pi$ is etale at $(O^\times)^*\times \{0\}$.

The full Hecke algebra $\H$ defines a quasicoherent sheaf
of algebras on $\Z$, we now describe its extension to a quasicoherent sheaf
of algebras $\Hb$ on $\Zb$. The latter depends on the choice of a maximal
open compact subgroup $K_0=SL(2,O)$.  Fixing this choice we set $\Ht^+=End_{T^+}(\mcS(G/U^+))$,
where $(G/U)^+=O^2\setminus \{0\}\subset F^2\setminus \{0\}=G/U$. We also let $\Ht=End_T(\mcS(G/U))$.
It is clear that $\Ht^+$ defines a quasicoherent sheaf on $\Zt^+$ whose restriction to the open 
subset $\Zt:=\Zt^+\setminus (O^\times)^*\times \{0\}$ is the quasicoherent sheaf defined by $\Ht$.

The action of $\H$ on $\mcS(G/U)$ defines a homomorphism
$\pi^*(\H)\to \Ht$ which is an isomorphism on a Zariski neighborhood of $\partial \Zt^+=\Zt^+\setminus \Zt$.
Thus we get a well defined quasicoherent sheaf of algebras $\Hb$ on $\Zb$ such that
$\Hb|_\Z=\H$ and the induced map $\pi^*(\Hb)|_\Zt\to \Ht$ extends to a map $\pi^*(\Hb)\to \Ht^+$
which is an isomorphism on a neighborhood of $\partial \Zt^+$. 

It is clear that $K_0^2$ acts on $\Hb$ and for an open subgroup $K\subset K_0$
the subsheaf $\Hb_K$ of $K^2$ invariants is a coherent sheaf of algebras.

We leave it to the reader to show that although $\Hb$ depends on an auxiliary 
choice, different choices lead to algebras which are canonically Morita equivalent. 
Thus we can consider the category of sheaves of nondegenerate $\Hb$-modules which can be checked
to be canonically equivalent to $\Smb$. If the subgroup $K\subset K_0$ is nice 
in the sense of \cite{centre} then for every component $X$  of $\Zb$   either the coherent sheaf of algebras
$\Hb_K|_X$ is zero or the corresponding summand in $\Smb$ (respectively, $\Smb^{qc}$) is canonically equivalent to 
the category of  coherent (respectively, quasicoherent) sheaves of  $\Hb_K|_X$-modules.
\end{example}

\subsection{Compactified center and a spectral description of the compactifed category}

Let $Z=Z(G)$  be the  Bernstein center of $G$
and $\Z=Spec(Z)$ be its spectrum. By the main result of \cite{centre} (the set of closed points of)
$\Z$ is in bijection with the set $Cusp(G)$ of {\em cuspidal data}, i.e. the set of $G$-conjugacy classes of pairs $(L,\rho)$, where
$L\subset G$ is a Levi subgroup and $\rho$ is a cuspidal irreducible representation of 
$L$. 

\subsubsection{Compactified center} Let  $\Zbar$  denote its
compactification described as follows. We have a
canonical isomorphism $\Z=\Ztil/W$ where $W$ is the Weyl group, and
$\Ztil$ parametrizes pairs $(L,\rho)$ where $L$ is a Levi subgroup
containing a fixed maximally split Cartan $T$ and $\rho$ is a cuspidal
representation of  $L$. 
The complex torus  $\LT={\mathcal X}(L)$ acts
nn the union $\Ztil_L$ of
components corresponding to a given Levi subgroup $L\supset T$;
here ${ \mathcal X}(L)$ stands for the group 
of unramified characters of $L$ acting on the set of representations by twisting. Notice that $\LT$ is a torus with $X^*(\LT)=L/L^0$.
The action is transitive on each component
and the stabilizer of each point is finite. The space
$X^*(\LT)_\RE=X_*(Z(L))_\RE$ (where $X_*$ stands for 
the lattice of $F$-rational cocharacters)
contains hyperplanes corresponding to
the roots of $Z(L)$ in $\g$; the fan formed by these hyperplanes
defines  an equivariant compactification $\overline{\LT}$ of
$\LT$. We set $$\Zbar =\wt{\Zbar}/W,\ \ \ \ \ \ \ \ \  \wt{\Zbar}=\cupl_L \overline{\LT} \times
^{\LT} \Ztil_L , $$ where the right hand side makes sense
because the action of $W$ on $\Ztil$ extends to the
compactification, here we use the notation $X\times^HY=(X\times Y)/H$.
Notice that every component of $\Ztil_L$ is of the form $^LT/A$ for a finite
subgroup $A\subset ^LT$, thus the corresponding component of
$\overline{\LT} \times
^{\LT} \Ztil_L $ is identified with $\overline{\LT}/A$.

For a parabolic $P=LU$ let $Z^0(L)\subset Z(L)$, $Z^+(L)\subset Z(L)$
be the subalgebras consisting of distributions supported on $L^0$ and $L_P^+$
respectively, set also $\Z^0(L)=Spec(Z^0(L))$, $\Z^+(L)=Spec(Z^+(L))$. 

It is clear that 
\begin{equation}\label{Z0}
\Z^0(L)= \Z(L)/{\mathcal{X}}(L),
\end{equation}
where ${\mathcal{X}}(L)$ is the group of unramified characters of $L$.

\begin{Prop}
\label{L1}
a) $\Zbar$ admits a canonical stratification indexed by conjugacy
classes of parabolic subgroups, where the stratum $\Zbar_P$
corresponding to the class of a parabolic $P$ is identified with $\Z^0(L)$.

b) The embedding $\Z^0(L)\to \Zbar$ canonically extends to a
 map $\Z^+(L)\to \Zbar$ which is etale on a Zariski neighborhood of $\Zbar_P\cong \Z^0(L)$. 

Given two parabolics $P\subset Q$ we have a canonical
map $c_P^Q:\Z^+(L_P)\to \Z^+(L_Q)$ which is compatible with maps to
$\Zbar$.

Moreover, for three parabolics $P_1\subset P_2\subset P_3$ we have
$$c_{P_1}^{P_3}=c_{P_2}^{P_3}c_{P_1}^{P_2}.$$

\end{Prop}

\begin{pf}
Let $\wt{\Zbar}_L =  \overline{\LT} \times
^{\LT} \Ztil_L $. 

It is a standard fact that  $\LT$-orbits in $ \overline{\LT}$ are in bijection
with parabolic subgroups containing $L$, so that the orbit $\overline{^LT}_Q$ corresponding
to a parabolic $Q=MU_Q$  is identified with ${\mathcal X}(L)/{\mathcal X}(M)$. 
The stratification of $\overline{\LT}$ by $\LT$-orbits induces a stratification  on 
$\wt{\Zbar}_L$, the stratum corresponding to a parabolic $Q$ will be denoted by $\wt{\Zbar}_L\{Q\}$.

Fix a conjugacy class $\bf{P}$ of parabolic subgroups and set $\wt{\Zbar}_L({\bf{P}})=
\cupl_{Q\in {\bf{P}} }\wt{\Zbar}_L\{Q\}$.
Let $\wt{\Zbar}_{\bf{P}}=
\cupl_L \wt{\Zbar}_L(\bf{P})$.

It is clear that $(\wt{\Zbar}_{\bf{P}})$ is a stratification of $\wt{\Zbar}$ and each stratum is
$W$-invariant. Thus $\Zbar_{\bf{P}}:= \wt{\Zbar}_{\bf{P}}/W$ are strata of 
a stratification of $\Zbar$.

The map $Q\mapsto \wt{\Zbar}_L\{Q\}$ is easily seen to be $W$-equivariant, it follows
that for a parabolic $P=LU\in \bf{P}$ we have 
$$\Zbar_{\bf{P}}\cong \cupl_{M, T\subset M \subset L} \wt{\Zbar}_M\{P\}/W_L.$$

The above isomorphism $\overline{^LT}_Q\cong {\mathcal X}(L)/{\mathcal X}(M)$
shows that 
$ \wt{\Zbar}_M\{P\}\cong \Zbar_M(L)/{\mathcal X}(L)$. Passing to the union over $M$
and taking quotient by the action of $W_L$ (which commutes with the action of ${\mathcal X}(L)$) we get
$\Zbar_{\bf{P}}\cong \Z(L)/{\mathcal X}(L)$ which yields (a) in view of
\eqref{Z0}.

To check (b) observe that for parabolic subgroups $Q=MU_Q\supset P=LU_P\supset T$ the cone 
$\RE^{\geq 0} \Lambda_Q^+$ belongs to the fan defining the toric variety $\overline \LT$.
Let $V_L\{Q\}$ be the corresponding affine open subset in $\overline \LT$ and ${\mathcal V}_L\{Q\}$
be the corresponding open affine in $\wt{\Zbar}_L$. Thus ${\mathcal V}_L\{Q\}$
is a Zariski open neighborhood of $\wt{\Zbar}_L\{Q\}$. 

It is easy to see that  ${\mathcal V}_L\{Q\}$ is $W_M$ invariant and ${\mathcal V}_L\{Q\}/W_M\cong \Z^+(L)$. 
Since $\Zbar=\wt{\Zbar}/W$,  claim (b) follows from the fact that the stabilizer of any point
$x\in \wt{\Zbar}\{Q\}$ is contained in $W_M$.

c) follows by inspection. 
\end{pf}

In order to relate $\Smb$ to $\Zb$ we will need the following general concept.
Let $X$ be an algebraic variety. By a {\em quasicoherent enrichment} of 
a category $\CC$ over $X$ we will mean
assigning to objects $M$, $N\in \CC$ an
object $\uHom(M,N)\in QCoh(X)$ together with an isomorphism $Hom(M,N)=\Gamma(\uHom(M,N))$
and maps $\uHom(M_1,M_2)\otimes_
{\O_{X}} \uHom(M_2,M_3)\to \uHom(M_1,M_3)$ satisfying the associativity constraint and 
compatible  with the composition of morphisms in $\CC$. 
If the quasicoherent sheaf $\uHom(M,N)$ is actually coherent for all $M,N\in \CC$ we say
that the enrichment is coherent.

\begin{Prop}\label{prop1}
The category $\Smb$ (respectively, $\Smb^{qc}(G)$) admits a natural lifts to a category coherent (respectively, quasicoherent) 
enrichment over $\Zbar$.
\end{Prop}

\begin{Cor}
The categories $\Smb$, $\Smb^{qc}$ split as a direct sum indexed by components of $\Z$. \qed
\end{Cor}

Before proceeding to prove the Proposition we state a general elementary Lemma.


\begin{Lem}\label{L2}
Let $X=\coprod X_i$ be a scheme with a fixed stratification (i.e. the closure of $X_i$ 
coincides with $\coprod _{j\leq i} X_j$ for some partial order $\leq $ on the set $I$ of strata).
Set $U_i=\coprod _{j\geq i} X_j$, this is an open subset of $X$.
Suppose that for each $i$ we are given a map $u_i:Y_i\to U_i$, such that 

i) $X_i\times _X Y_i\iso X_i$ 

ii) $u_i$ is etale over a Zariski neighborhood of $X_i$.

iii) For $j\leq i$ set $Y_{ji}=Y_j\times _X U_i$.
Then the map $Y_{ji}\to U_i$ factors through a map $u_{ji}:Y_{ji}\to Y_i$. 
Moreover, for $k<j<i$ the map $Y_{ki}\to U_j\supset U_i$ factors through a map
$u_{kji}: Y_{ki}\to Y_{ji}$.

Let $Y_d=\coprod_{i_1<\dots < i_d} Y_{i_1i_d}$.
Then the diagram $Y_3 \mathrel{\substack{\textstyle\rightarrow\\[-0.6ex]
                      \textstyle\rightarrow \\[-0.6ex]
                      \textstyle\rightarrow}} Y_2
                      \rightrightarrows Y_1\to X$
 satisfies descent for quasicoherent sheaves, i.e. 
$QCoh(X)$ is equivalent to the category of quasicoherent sheaves on $Y_1$ with isomorphisms
of the two pull-backs to $Y_2$ whose pull-backs to $Y_3$ satisfy the natural compatibility. 

 
 
 
 


\end{Lem}

\begin{Rem}
To fix ideas let us first prove the Lemma for sheaves in the analytic topology assuming we work over the base field  $\Ce$. Then it is easy to see that we can find an open subset $Y_i^o
\subset Y_i$ for each $i$ so that $Y_i^o$ maps isomorphically to a neighborhood of $X_i$
in $X$. Moreover, we can arrange it so that the images of $Y_i^o$ and $Y_j^o$ have a nonempty 
intersection only if $i\leq j$ or $j\leq i$. Replacing $Y_i$ by $Y_i^o$ does not affect the category
of gluing data; however, $(Y_i^o)$ is just an open covering of $X$, so the claim is clear.
\end{Rem}

\subsubsection{Proof of Lemma \ref{L2}} Let $X_0$ be a closed stratum. Running an inductive argument, we can assume the theorem
is known for the stratified space $X\setminus X_0$. Then we are reduced to proving the claim in the situation
when the stratification consists of two strata $X=X_0\coprod X_1$. Replacing $Y_0$ by its open subset
containing $X_0$ clearly does not affect the category of descent data, so we can assume without
loss of generality that $Y_1\to X$ is etale. By a standard argument the claim reduces to exactness
of the complex of sheaves on $X$:
$$0\to \O\to (Y_0\to X)_*(\O) \oplus (Y_1\to X)_*(\O)\to (Y_{01}\to X)_*(\O)\to 0.$$
The complex is clearly exact over $X_1$, so it is enough to show that local cohomology of this complex
with support on $X_1$ vanishes. This reduces to showing that $\alpha^!(\O)\iso \beta^!(\O)$,
where $\beta:X_0\to X$, $\alpha_1:X_0\to Y_0$. This follows from conditions (i), (ii). \qed 

\subsubsection{Proof of Proposition \ref{prop1}} 
Proposition  \ref{L1} implies that $X=\Zbar$ with the stratification of Proposition \ref{L1}(a), 
and $Y_i=\Z(L_i^+)$  satisfy the conditions of Lemma \ref{L2}. It is easy to see from  the definition of 
$\Smb^{qc}$ that the collection of quasi-coherent sheaves $Hom(M_P,N_P)$ provides gluing
data described in Lemma \ref{L2}. Also, for $M,\, N\in \Smb$ the module $Hom(M_P,N_P)$ is
finitely generated, so the resulting quasicoherent sheaf is in fact coherent. $\qed$

\subsection{The spectral description of $\Smb$}
\label{sp_des}

Recall that for every component $X\subset \Z$ the choice of a
sufficiently small nice (in the sense of \cite{centre}) open compact subgroup $K\subset G$ defines a
coherent sheaf of algebras $\A=\A_X(K)$ on $X$ with an equivalence between
$\A_X(K)$ modules and the corresponding summand in smooth $G$-modules.

Let $\fB$ be the (reductive) Bruhat-Tits building of $G$. We fix a special vertex $x\in \fB$
and let $K_x\subset G$ denote the corresponding maximal compact subgroup.

We also fix a maximally split Cartan subgroup $T$, such that the corresponding apartment $A_T\subset \fB$ contains $x$.
Thus $A_T$ is an affine space with underlying vector space
$V=X_*(T)\otimes \RE$. 
Let $V_+\subset A_T$ be  a Weyl cone with vertex at $x$
and  $B$ the corresponding minimal parabolic subgroup.
 

It follows from the Iwasawa decomposition (see e.g. \cite[\S 3.3.2]{Tits}) that we have a natural
bijection 
\begin{equation}\label{Iwa}
P'\backslash G/K_x \cong \check{\Lambda}_P,
\end{equation}
where $P'$ denotes the commutator group. Let $G^+(x,P)$ be the union of cosets corresponding to 
elements in $\check{\Lambda}_P^+\subset \check{\Lambda}_P$. 

Let $(U_P\backslash G)^+(x)$ be the image of $G^+(x,P)$ in $U_P\backslash G$ and $\mcS^+(U_P\backslash G)(x)\subset
\mcS(U_P\backslash G)$ be the subspace of functions whose support is contained in $(U_P\backslash G)^+(x)$.

\begin{Prop} 
a) The left action of $L$ on $\mcS(U_P\backslash G)$ restrict to an $L^+_P$ action on
$\mcS^+(U_P\backslash G)(x)$. 

b) Let $K\subset K_x$ be an open subgroup. The $\H(L_P^+)$-module $\mcS^+(U_P\backslash G)(x)^K$
is finitely generated and projective.

c) There exists a unique object $\PP(x)^K\in \Smb$ such that $\PP(x)^K_P=\mcS^+(U_P\backslash G)(x)^K$
while the isomorphism \eqref{isom_struk} for $P=G$ comes from the natural arrow  $\mcS^+(U_Q\backslash G)(x)^K \to \mcS(U_Q\backslash G)^K
=J^G_Q (\mcS(G)^K)$.

d) For every open subgroup $K\subset K_x$ the object $\PP(x)^K $ is locally projective, 
i.e. the functor $\uHom(\PP(x)^K,\ \ )$ is exact.

e) 
For every component $X$ there exists an open subgroup $S\subset K_x$ such that for any open subgroup $K\subset S$
the object $\PP_x^K$ is
local generator of the corresponding summand $\Smb_X\subset \Smb$. The latter property means that $\uHom(\PP_\sigma(x)^K, \ \ )$ is conservative
i.e. it $\uHom(\PP_\sigma(x)^K, \MM)\ne 0$ for $0\ne\MM \in \Smb_X$.
\end{Prop}

\proof a) Bijection \eqref{Iwa} intertwines  the left action of $L$ on $P'\backslash G/K_x$  with the action of  $\check{\Lambda}_P= L/L^0$ on itself by translations,
this implies part a). 
The space $\mcS^+(U_P\backslash G)(x)^K$ splits as a direct sum indexed by $P\backslash G/K$, each summand is isomorphic to 
the space $\mcS(L_P^+)^{K_L}$ for some open compact subgroup $K_L\subset K$, this implies b).

Notice that given modules $M_P$ and isomorphisms \eqref{isom_struk} for $P=G$ as in Definition \ref{sm_def}, the rest of the isomorphisms
\eqref{isom_struk} satisfying the requirements of the Definition are defined uniquely (if they exist) provided that each module
$M_P$ is torsion free as a module over $Z^+(L_P)$. This implies uniqueness in c). Existence follows from the fact that
 $G^+(x,P)\subset G^+(x,Q)$ for parabolic subgroups $P\subset Q$.

Statement d) follows from b). 

Finally e) follows from the corresponding statement about $\H_K$ established in \cite{centre}. 
\qed

 Set $\bA_X=\bA_X(x,K):=\uHom(\PP(x)^K,\PP(x)^K)|_X$, this is a sheaf of algebras on $X$.

 \begin{Cor}\label{algAx} Given a component $X\subset \Zbar$ for any small enough 
 open subgroup $K\subset K_x$ 
 we have a canonical equivalence between the category 
 of (quasi)coherent sheaves of $\bA_X(x,K)$-modules and the summand in $\Smb$
(respectively, $\Smb^{qc}$) corresponding to $X$.
 \end{Cor}


\subsection{Compactified category and filtered modules}\label{filmod}
For a simple root $\alpha$ let $P_\alpha$ be the corresponding maximal proper parabolic
and $\Zb_\alpha$ be the closure of the corresponding stratum in $\Zb$. It is easy to see
that $\Zb_\alpha$ is a divisor. For a weight $\lambda$ set $D_\la=\sum_\alpha \langle \check{\alpha}
,\la\rangle [\Zb_\alpha]$ where the sum runs over the set of simple roots, let also $\O(\la)=\O_{\Zb}(D_\la)$
and $\F(\lambda)=\F\otimes_{\O(\Zb)}\O(\la)$.

A coherent sheaf $\F$ on $\Zbar$ is determined by the graded module $\oplus _\la \Gamma( \F(\la))$,
over the homogeneous coordinate ring  $\oplus \Gamma(\O(\la))$.

 If $\F$  is torsion free then the natural map $\F(\mu)\to \F(\la+\mu)$, $\la\in \Lambda^+$ is injective
 and $\bigcup\limits_\la \F(\la)=j_*j^*(\F)$ where $j:\Z\to \Zb$ is the embedding.
 Thus the category of torsion free coherent sheaves $Coh_{tf}(\Zb)$ admits
 a full embedding into the category of $\O(\Z)$-modules equipped with a filtration indexed 
 by $\Lambda^+$ compatible
 with the natural filtration on the ring $\O(\Z)$: to a sheaf $\F$ it assigns the module $\Gamma(j^*(F))$
 with the filtration by the subspaces $\Gamma(\F(\la))$. 
 
 Applying it to the sheaf of rings $\bA_X(x,K)$ we get a filtration on the Hecke algebra 
 $F_{\leq \lambda}^{spec}(\H_K)$.
  
Thus we obtain the following:

\begin{Prop}
We have a full embedding from the category  $\Smb^{tf}_K$
 of torsion free objects in $\Smb_K$ to the category of modules
 over $\H_K$ equipped with a filtration compatible with the filtration
 $F^{spec}$ on $\H_K$.
 \end{Prop}

 It is clear that  the associated graded of the filtered module
 in the image of such an embedding is finitely generated;
 recall that a filtration with this property is called a {\em good filtration}.

We also have the left adjoint functor $Loc$ from the category of 
 $\H_K$ modules with a good filtration to $\Smb_K$.

\subsubsection{A geometric description of the filtration} 
  We now provide a more explicit description of the filtration
  $F_{\leq \lambda}^{spec}(\H_K)$.

 Recall that the two sided cosets of $K_x$ in $G$ are indexed by $X_B^+$ (see e.g.
 \cite[\S 3.3.3]{Tits}),  for $\la\in X_B^+$
 let $G_\la$ denote the corresponding coset.
 Let $F^{geom}_{\leq \la}(\H_K)$ be the space of functions whose support
 is contained in $\cupl_{\mu \leq \la}G_\mu$. 
  
 \begin{Prop} 
\label{fifi} 
 For every 
 open compact $K\subset K_x$ there exists
 $\la_0\in X_B^+$ such that:
\begin{equation}\label{Fspecgeom}
F_{\leq \lambda}^{spec}(\H_K) =\{h\in \H_K\ |\ \forall \mu \in \la_0+X_B^+ :\ \  h*F^{geom}_{\leq
\mu}(\H_K)\subset F^{geom}_{\leq \mu+\la}(\H_K).\}
\end{equation}
 \end{Prop}
 
 {\em Proof\ } It follows from the definition that 
 
  \begin{equation}\label{Fspec}
 h\in F_{\leq \lambda}^{spec}(\H_K)\iff \forall P,\mu \ :\ \mcS((U_P\backslash G)_{\leq \mu})*h\subset  \mcS((U_P\backslash G)_{\leq {\bar{\la}_P} +\mu}),
  \end{equation}
 where ${\bar{\la}_P}$ is the image of $\lambda$ in $\check{\Lambda}_P$ and $(U_P\backslash G)_{\leq \mu}= \cupl_{\nu\leq \mu} (U_P\backslash G)_\nu$
 for the standard partial order $\leq $ on $\check{\Lambda}_P$; here $(U_P\backslash G)_\nu$ is the image of the coset corresponding
 to $\nu$ under bijection \eqref{Iwa}.
  )
 Let $X_P=(G/U_P\times G/U_{P^-})/L$. 
We have $K_x\backslash X_P/K_x
\cong \check{\Lambda}_P,$ let $(X_P)_\mu$ denote the two-sided coset corresponding to $\mu\in \check{\Lambda}_P$.  
 Then the  condition in the right hand side of \eqref{Fspec} is equivalent to:
 \begin{equation}\label{allPmu}
 \forall P,\mu \ :\ h*\mcS((X_P)_{\leq \mu})\subset \mcS((X_P)_{\leq \la +\mu}),
 \end{equation}
 where $(X_P)_{\leq \mu}= \cupl_{\nu\leq \mu} (X_P)_\nu$.
 this is clear by considering the projection $X\to G/P^-$ with fiber $G/U_P$.
 Assume now that $h$ lies in the set in the right hand side of \eqref{Fspecgeom}.
 Applying the map $B_I$ of \cite[Definition 5.3]{BK}  and using \cite[Lemma 5.5]{BK},
 we see that $h$ also satisfies \eqref{allPmu}. This shows that the right hand side of \eqref{Fspecgeom}
 is contained in the left hand side. 
 We proceed to check the opposite inclusion. The Rees ring 
 $\oplus_\lambda F_{\leq \lambda}^{spec}(\H_K)$
is finite over its center which is a finitely generated commutative ring, hence the Rees ring  is finitely generated.
 Thus it suffices to check that for a finite set of generators $h_i\in F_{\leq \lambda_i}^{spec}(\H_K)$
 we have:
 $$h_i *F^{geom}_{\leq
\mu}(\H_K)\subset F^{geom}_{\leq \mu+\la_i}(\H_K) \ \ \forall \mu\in \la_0+\La^+$$
for some $\lambda_0\in \Lambda^+$. Existence of such a $\lambda_0$ for a given
$h_i$ follows from \cite[Lemma 5.5]{BK}. Since we consider a finite set of $h_i$,
there exists $\lambda_0$ which satisfies the requirement for all $h_i$.
   \qed
 )
\begin{Rem}
It easily follows from the construction that $gr(F^{spec})$ is a Noetherian ring.

Notice that the  geometric filtration on $\H$ is also compatible with the algebra structure;
however, its associated graded is neither  Noetherian, nor finitely generated in general. 
This is closely related to the fact that the intertwining operator acting on the
space of functions over $\underline{(G/U)} (O/\pi^nO)$ is not an isomorphism for $n>1$
(here $\underline{(G/U)}$ is a scheme over $O$ coming from the $O$-group scheme with generic
fiber $G$ corresponding to $x\in \fB$).
\end{Rem}

\begin{example}\label{int_bim}
In view of Proposition \ref{fifi} the geometric filtration $F^{geom}$ on $\H$
makes it into a filtered  
module over the filtered algebra $(\H,F^{spec})$.
Applying the functor $Loc$ to that filtered module we get an object
in $\Smb$.
 We denote it by $\Hb'$ and call it the  {\em intertwining}  object in $\Smb$. 
 \end{example}


\begin{Lem}
\label{XYLem} 

%
%
Assume that $G=SL(2)$. 

The intertwining object $\Hb'$ is equivalently described by $\Hb'_G=\H$,
$$\Hb'_B=\{f\in \mcS ( G/U)\  | \ supp(I^{-1}(f))\subset (G/U_-)^+\},$$
where $I^{-1}$ denotes the inverse {\em intertwining operator} taking
values in functions of bounded support.\footnote{Here by a bounded set
in $G/U_-=F^2\setminus \{0\}$ we mean a subset with compact
closure in $F^2$.}
\end{Lem}

\proof 
Without loss of generality we can assume that $x$ is the standard vertex, so
that $K_x=SL(2,O)$, we will write $K_0$ instead of $K_x$.

We need to check that for $h\in \H_K$ and large $\lambda\in \Lambda$
we have:
$$ h\in F^{geom}_{\leq \la} \ \iff supp( I_P^{-1}A_P(h))\subset (X_P)_{\leq \la},$$
where $X_{\leq \la}$ is the union of two-sided $K_0$ cosets corresponding
to weights $\mu\leq \la$ and $A:\H\to \mcS(Y_P)$. By \cite[Theorem 7.6]{BK}
 $I^{-1}A=B^\star$ (notations of {\em loc. cit.}). It follows from \cite[Lemma 5.5]{BK}
 that for large $\lambda$ we have 
 $$supp(h)\subset G_{\leq \la} \iff supp (B^\star(f))\subset
 X_{\leq \la}.$$ The claim follows. \qed

\begin{Rem}
A similar statement can also be checked for an arbitrary reductive 
group $G$ based on a generalization of \cite[Theorem  7.6]{BK}
to an arbitrary parabolic subgroup. 
\end{Rem}

\begin{Rem}\label{int_rem}
We expect the local trace formula \cite{Ar} to be closely related to 
the computation of Chern character of $\Hb'$ taking values
in the appropriate Hochschild homology group. 
\end{Rem}

\section{Hochschild homology and character values}
 Recall the notion of {\em Hochschild homology}
of an abelian category \cite{Ma}, \cite{Ke1}.
For a coherent sheaf of algebras $A$ over a quasi-projective algebraic variety $X$ over a field  
we have $HH_*(X;A)\cong HH_*(A-mod),$
where $A-mod$ is the abelian category of coherent sheaves of $A$-modules and
$HH_*(X;A)$ is defined as the derived global section of
a naturally defined object 
 $R\underline{Hom}_{A\otimes A^{op}}(A,A)$ in the derived category of sheaves on $X$,
 here the isomorphism is shown in \cite{Ke}.

One can also define the {\em compactly supported}
Hochschild homology $HH_*^c(X;A)$ as the derived  global sections with compact support in the sense of 
\cite{D} of $R\underline{Hom}_{A\otimes A^{op}}(A,A)$.
If $X$ is projective then we also have 
$$HH_*^c(X;A)\cong HH_*(X;A)\cong HH_*(A-mod),$$ where the first isomorphism
is clear since $R\Gamma=R\Gamma_c$ for a projective scheme.
Also, for an open subscheme $U\subset X$ we have a natural push forward map $HH_*^c(U;A|_X)\to HH_*^c(X,A)$.


Assume that the sheaf of algebras $A$ (locally) has finite homological dimension.
Then we have the {\em Chern character} map $ch:K^0(A-mod)\to HH_0(A-mod)$,
see e.g. \cite[\S 4.2]{Ke} for the definition (it is called the Euler class map in {\em loc. cit.}). 

\begin{example}\label{ExHoch}
For future reference we spell out this general construction in some simple special cases.
We leave the proofs of these standard facts to the interested reader.

\begin{enumerate}
\item Assume $X$ is affine, so that $A-mod$ is the category of finitely generated modules
over a Noetherian ring which we also denote by $A$. Then $HH_*(X;A)=HH_*(A)$ is computed
by the bar complex of $A$, in particular
$HH_0(X;A)=HH_0(A)=A/[A,A]$.
A finitely generated projective
module $M$ is isomorphic to $A^{\oplus n}e$ for an idempotent $e\in Mat_n(A)$ for some $n$.
Then $ch(M)=\sum (e_{ii}) \mod [A,A]$.

\item Assume that $X=U_1\cup U_2$ for affine open subsets $U_1$, $U_2$. 
Let $A_i=\Gamma(U_i,A)$, $i=1,2$ and $A_{12}=\Gamma(U_1\cap U_2, A)$.
Then $HH_*(X;A)$ is computed by the complex
$$ Cone\left( Bar(A_1)\oplus Bar(A_2)\to Bar(A_{12}) \right)[-1].$$
For a locally projective module coherent sheaf $M$ of $A$-modules we can find
integers $n$, $m$ and idempotents $e^1\in Mat_n(A_1)$, $e^2\in  Mat_m(A_2)$
together with isomorphisms $\Gamma(U_1,M)\cong A_1^{\oplus n}e^1$,
$\Gamma(U_2,M)\cong A_2^{\oplus m}e^2$ and matrices $a\in Mat_{nm}(A_{12})$,
$b\in Mat_{mn}(A_{21})$ which induce the inverse isomorphisms 
between $A_{12}^{\oplus n}e^1$ and $A_{12}^{\oplus m}e^2$
coming from the identification of both with $\Gamma(U_1\cap U_2,M)$. 
In this setting $ch(M)$ is represented by the cocycle $(\sum e^1_{ii}, \sum e^2_{jj},
\sum a_{ij}\otimes b_{ji})\in A_1\oplus A_2\oplus A_{12}^{\otimes 2}$.

\item Let $U\subset X$ be an affine open subset such that its complement $Z$ is also affine.
Let $A_1=\Gamma(U,A)$, and let $A_1$, $A_{12}$ be the algebras of sections of 
$A$ on the formal neighborhood and the punctured formal neighborhood of $Z$
respectively. Then the statements of part (b) continue to hold {\em mutatis mutandis}. 
\end{enumerate}
\end{example}

\begin{definition}
In the above setting, for an object $\MM\in A-mod$ with {\em proper} support we have the Chern
character $ch^c(\MM)\in HH_0^c(X;A)$ as follows. Fix a proper subscheme
$Z\subset X$ containing support of cohomology of $\MM$. 
Then \cite[\S 5.7]{Ke} yields $ch_Z(\MM)\in R\Gamma_Z(X,R\underline{Hom}_{A\otimes A^{op}}(A,A))$.
We define $ch^c(\MM)$ to be the image of $ch_Z(\MM)$ 
under the canonical map 
$R\Gamma_Z(X,F)\to R\Gamma_c(X,F)$, $F=R\underline{Hom}_{A\otimes A^{op}}(A,A)$.
\end{definition}

It is immediate to check that $ch^c(\MM)$ is independent of the auxiliary choice of $Z$. 

In particular, we have Chern character map $\bar{ch}:K^0(\Smb)\to HH_0(\Smb)$ and $ch^c:K^0(Adm)\to HH_0^c(Sm)$,
where $Adm\subset Sm$ is the subcategory of admissible modules.

Let $\KK^c\subset \KK$, $\KK^{nc}\subset \KK$ be the subspace of measures supported on compact 
(respectively, noncompact) elements.

\begin{Conj}\label{char_form_conj}
a) We have a canonical isomorphism:

 $ \KK_G^c\cong Im(HH_0^c(Sm)\to HH_0(\Smb))$.

b) For $\rho\in Adm$ and $g\in G$ we have $WO_g(\bar{ch}(\rho))=\chi_\rho(g)$
if $g$ is compact regular semisimple. (Here we  use
that $\bar{ch(\rho)}$ is the image of $ch^c(\rho)$ thus it 
belongs to  $Im(HH_0^c(Sm)\to HH_0(\Smb))$ which we identify with
$ \KK_G^c$ by a). 
\end{Conj}

The proof of Conjecture for $G=SL(2)$ is presented in the next section; we plan
to present the proof in the general case in a later publication.

\section{$SL(2)$ calculations}

 \begin{Thm}\label{char_form}
Conjecture \ref{char_form_conj} holds for $G=SL(2)$.
\end{Thm}

The rest of the section is devoted to the proof of the Theorem. From now on 
set $G=SL(2)$.

\subsection{Explicit complexes for Hochschild homology}\label{EHo}
Applying the general construction of section \ref{sp_des} with $\sigma=\{x\}$
where $x$ is the vertex with stabilizer $G(O)$ we arrive at a sheaf of algebras
$\A$ on $\Zb$ such that  a direct summand in $\Smb$ is canonically identified with the category
of coherent sheaves of $\A$-modules. It is easily seen
to coincide with the sheaf $\Hb_K$ introduced in Example \ref{sl2ex}.
We keep notations of that Example.

We also let $\Hh^+_K$ denote its
sections on the formal neighborhood $\wh{\partial \Z}$ of $\partial \Z=\Zb\setminus \Z$ and
$\Hh_K$ the sections on the punctured formal neighborhood of $\partial \Z$.
We set $\Hh^+= \cupl_K \Hh^+_K$
 and $\Hh=\cupl_K \Hh_K$. 
  %
We also let $X=(G/U\times G/U^-)/T$ be the set of  rank one matrices in $Mat_2(F)$ and
$X_+=X\cap Mat_2(O)$.

Then we have $\Hh^+_K=End_{T^+}(\mcS(G/U)_+^K)\otimes_{\Ce[t]}\Ce[[t]] $,
$\Hh_K=\Hh_+^K \otimes_{\Ce[[t]]} \Ce((t))$. We have $\Hh_K=End_{\Hh(T)}(\mcS_b^K(G/U))$,
where $\mcS_b(G/U)^K\cong \mcS(G/U)\otimes _{\Ce[t,t^{-1}]}\Ce((t))$ is the space of functions
on $G/U$ with bounded support.

Also notice that $\Hh$ (though not $\Hh_+$)
carries a $G\times G$ action. 

\begin{lem}\label{complexes}
a) $HH_*(\Smb)$ is computed by the complex 
\begin{equation}\label{BarSmb}
Bar(\Smb):= cone[Bar(\H)\oplus Bar(\Hh^+)\to Bar(\Hh)][-1].
\end{equation}

b) $HH_*^c(Sm)$ is computed by the complex
\begin{equation}\label{BarSmc}
Bar^c(Sm):=cone[Bar(\H)\to Bar(\Hh)][-1].
\end{equation}

c) $HH_*^c(Sm)$ is canonically isomorphic to 
derived $G$ coinvariants in the two-term complex 
\begin{equation}\label{HtoHhat}
\H\to \Hh
\end{equation}
(placed in degrees 0,1).

\end{lem}

\proof For a complex $\F$ of coherent sheaves on $\Zb$
the complexes 
$$cone[\Gamma(\Z,\F)\oplus \hat{\Gamma}^+(\F) \to \hat{\Gamma}(\F)][-1],$$
$$cone[\Gamma(\Z,\F)\to \hat{\Gamma}(\F)][-1]$$
compute, respectively, $R\Gamma(\F)$ and $R\Gamma_c(\Z,\F)$, where
$\hat{\Gamma}^+(\F)$ and $ \hat{\Gamma}(\F)$ denote, respectively, 
sections of $\F$ on the formal neighborhood and on the punctured formal
neighborhood of $\partial \Z$. Applying this to a complex representing
$R\underline{Hom}_{\Hb_K\otimes \Hb_K^{op}}(\Hb_K,\Hb_K)$
and observing that we have canonical quasi-isomorphisms $\hat{\Gamma}^+(\F)\to Bar(\Hh^+)$,
 $\hat{\Gamma}(\F)\to Bar(\Hh)$ we get statements (a,b). 
 
 Statement (c) follows from isomorphisms $HH_*(M)\cong H_*(G,M)$,
 $HH_*^{\Hh}(N)\cong HH_*^\H(N)$ for an $\H$-bimodule $M$ and an $\Hh$-bimodule
 $N$. \qed

\subsection{Calculation of $HH_0$}
In this subsection we prove part (a) of the Theorem.

\begin{Prop}\label{SL2prop}
a) We have a short exact sequence:

\begin{equation}\label{KKGses} 
0\to \KK_G\to HH_0^c(Sm)\to \wh{\mcS(T)}/\mcS(T)\to 0.
\end{equation}

b) We have a natural isomorphism $\KK^c_G\cong Im(HH_0^c(Sm)\to HH_0(\Smb))$.

\end{Prop}

\bigskip

We start the proof with the following

\begin{Lem}
 $G$ acts trivially on the cokernel of the map $H\to \mcS(Y_\Delta)$. 
 \end{Lem}

{\em Proof}  follows from Lemma \ref{SL2ex}. \qed

\begin{Cor}\label{corH1} We have:

$H_1(G,CoKer(H\to \mcS(Y_\Delta)))=0$.

$H_1(G,CoKer(H\to \wh{\mcS(Y_\Delta)}))= \wh{\mcS(T)}/\mcS(T)$,

\noindent where $\wh{\mcS(T)}$ is the Tate completion of $\mcS(T)$ (functions
on the torus). 
\end{Cor}

\bigskip

{\em Proof} of Proposition \ref{SL2prop}. 
By Lemma \ref{complexes} (c),
 $HH_*^c(Sm)$ is identified with the derived $G$-coinvariants
in the complex $\H\to  \wh{\mcS(Y)}$
which is clearly isomorphic to \eqref{HtoHhat}. 
We have a short exact sequence 
$$0\to \mcS(Y_0)\to \mcS(Y)\to \mcS(Y_\Delta)\to 0,$$
where $Y_0=Y\setminus Y_\Delta$ is the open set of non-colinear  
pairs of vectors. Thus the action of $G$ on $Y_0$ is almost free, i.e.
the stabilizer of every point in $\{\pm 1\}$, and the orbits of $G$ on
$Y_0$ are indexed by the finite set $F^\times/(F^\times )^2$.
This shows that $H_1(G,\mcS(Y_0))=0$ and $H_0(G,\mcS(Y_0))$
is finite dimensional. Denoting as above $\wh{\mcS(Y)}=\mcS(Y)\otimes
_{\Ce[\varpi,\varpi^{-1}] }\Ce((\varpi))$ and similarly for
$\wh{\mcS(Y_0))}$ we see $H_i(G,\mcS(Y_0))=0$ for all $i$.
Thus derived  $G$ coinvariants of the complex \eqref{HtoHhat}
are canonically isomorphic to derived $G$-coinvariants
of the quotient complex $\fC= \H\to  \wh{\mcS(Y_\Delta)}$.
Notice that $H^0(\fC)=\KK$, 
 $H^1(\fC)=CoKer(\H\to \wh{\mcS(Y_\Delta)}))$ and $\fC$ 
is quasi-isomorphic to the complex with zero differential since the category
of $G$-modules has homological dimension one. Thus part(a)
of the Proposition follows from Corollary \ref{corH1}.

Lemma  \ref{complexes} shows that $Bar^c(Sm)$ is canonically
quasiisomorphic to $\Ce \Lotimes_G \fC$. One the other
hand, we have a short exact sequence of complexes
$$0\to Bar^c(Sm)\to Bar(\Smb) \to Bar(\wh{H(T^+)})\to 0,$$ which yields
a long exact sequence on cohomology:

$$\cdots \to HH_{i+1}(\wh{H(T^+)})\to HH_{i}^c(Sm)\to HH_i (\Smb) \to \cdots$$

We are interested in $i=0$. We have $HH_1(\wh{H(T^+)})\cong \wh{H(T^+)}$.
To finish the proof we need another

\begin{Lem}
a) Let us identify 

$$HH_1(\wh{H(T)})=H_1(G,\wh{\mcS(Y_\Delta)})=\wh{\mcS(T)}.$$

Then the image of the natural map $HH_1(\wh{H(T^+)})\to 
HH_1(\wh{H(T)})$ coincides with $\varpi \wh{\mcS(T^+)}$, i.e. with the space of functions supported on $\varpi T^+$,
where $\varpi\in F$ is a uniformizer.

b) In view of \eqref{KKGses}, part (a) yields a map $\varpi \wh{\mcS(T^+)}\cap \mcS(T)\to \KK_G\subset HH_0^c(Sm)$.
This map induces an isomorphism onto $(\KK_{nc})_G$ where
$\KK_{nc}\subset \KK$ is the space of measures supported on noncompact
elements.
\end{Lem}

{\em Proof of the Lemma.}  By Hochschild-Kostant-Rosenberg, 
we can identify Hochschild homology with forms. It is easy to see
that elements in $\mcS(T^0)\subset \wh{\mcS(T)}$ (where $T^0$ is the maximal 
compact subgroup in the torus $T$) locally on each component
of Bernstein center are proportional to $\frac{dz}{z}$ for 
a local coordinate coming from a global coordinate on $\Ce^\times$.
On the other hand, the image of $HH_1(\wh{H(T^+)})$ equals
$\Ce[[z]] dz$, this proves (a).

To prove (b) observe that the map in question can be described as follows.

We have a short exact sequence of $G$-modules
$$0\to \KK \to \H\to I\to 0,$$
where $I=Im (\H\to \mcS(Y_\Delta))$
which yields the Bockstein homomorphism $\phi: H_1(G,I)\to H_0(G,\KK)$.
By Corollary \ref{corH1} we have $H_1(G,I)\cong H_1(G,\mcS(Y_\Delta))=H_1(G,\mu_{\Pone})\otimes
\mcS(T)=\mcS(T)$. In view of (a) it suffices to check that $\phi: \mcS(\varpi T^+) \iso H_0(G,\KK^{nc})$.

We have a direct sum decomposition $\H=\H^c\oplus \H^{nc}$ compatible with the
decomposition $\mcS(Y_\Delta)=\mcS(Y_\Delta^c)\oplus \mcS(Y_\Delta^{nc})$, where 
$Y_\Delta^c = \Pone \times O^\times$ and $Y_\Delta^{nc}=Y_\Delta\setminus Y_\Delta^{nc}$.
Let us further decompose $Y_\Delta^{nc}=Y_\Delta^+\cup Y_\Delta^-$, where 
$Y_\Delta^+=\Pone \times (\varpi O\setminus \{0\})$, $Y_\Delta^-=\Pone \times (F\setminus O)$.

The image of $\phi$ clearly coincides with the image of $H_1(G,\mcS(Y_\Delta^+))$ under
the map coming from the class in $Ext^1 (\mcS(Y_\Delta^+), \KK)$ induced from the above 
short exact sequence. It is clear that it is contained in $H_0(G,\KK^{nc})$. It
coincides with $H_0(G,\KK^{nc})$ for the following reason.
Consider a regular orbit $O=G/T$, then the image $\KK_O$ of $\KK$ in 
$\mcS(O)$ is the kernel of the pushforward map $\mcS(G/T) \to \mu_{G/B}\oplus \mu_{G/B^-}$.
Then $(\KK_O)_G$ is one dimensional and the Bockstein map
$H_1(G,\mu_{G/B})\to H_0(G,\KK_O)$ is nonzero. \qed

\bigskip

The Lemma clearly implies statement (b) of Proposition. \qed

\subsection{Cocycles for Chern character and Euler characteristic}
Let now $M$ be a locally projective object of $\Smb$.

Thus there exist  idempotents $e\in Mat_n(\H)$, $e_+\in Mat_m(\Hh_+)$
and $a,b\in Mat_{n,m}(\Hh)$, such that $ab=e$, $ba=e_+$ with isomorphisms 
$M|_\Z\cong \H^n e$, $M_{\wh{\partial \Z}}\cong \Hh_+^m e_+$, so that $a$, $b$ 
induces the two inverse isomorphisms between the restrictions of 
$\H^n e$ and $ \Hh_+e_+$  to the punctured formal neighborhood of $\partial \Z$
coming from the identification with the restriction of $M$.

\begin{lem}\label{ch_coc}
In the above notation, $\bar{ch}(M)$ is represented by the cocycle for $Bar(\Smb)$
$c=(c_\H, c_+, \hat{c})$, where

$c=\suml_i e_{ii}\in \H$, $c_+=\sum_j (e_+)_{jj}\in \Hh_+$,
$\hat{c}=\suml_{i,j} a_{ij}\otimes b_{ji}\in \Hh\otimes \Hh$. 
\end{lem}

\proof This is a special case of Example \ref{ExHoch}(c). \qed

To state the next result we need the following notation. Fix $g\in G(O)$
normalizing $K$ 
thus $g$ acts by conjugation on the sheaf of algebras $\A$ and
on global sections of an object in $\A-mod$.

We introduce a zero-cochain $\tau_g$ 
for the dual complex $Bar(\Smb)^*$, $\tau_g=(\tau_\H(g), \tau_+(g),
\hat{\tau}(g))$, where $\tau_\H(g)=WO_g\in \H^*$ (the weighted orbital integral),
$\tau_+(g)=0\in \Hh_+^*$ and $\hat{\tau}(g):\alpha \otimes \beta \mapsto
Tr([\alpha,\Pi]\beta\circ g, \wh{\mcS(X)})$, where $\alpha,\beta\in \Hh$ act on $\wh{\mcS(X)}$ by left multiplication, 
$\Pi$ is the operator of multiplication by the characteristic function of $X_+$
acting on $\wh{\mcS(X)}$ and $g$ is acting on the right.


\begin{Prop}\label{tau_prop} a) 
The cochain $\tau_g$ is a cocycle. It satisfies: 
$$Tr(g,RHom (\Hb', \MM))=\langle \tau_g, \bar{ch}(M)\rangle,$$
for $M\in \Smb$, where 
 $\Hb'$ is the intertwining bimodule introduced at the end of section
\ref{compcat}.

b) For $\phi\in \KK_G^c = Im(HH_0^c(Sm)\to HH_0(\Smb))$ we have:
 $$\langle \phi, \tau_g\rangle = WO_g(\phi).$$

c) The image of  $\tau_g$ in $ HH_0^c(Sm)$  is independent of the auxiliary
choices in the definition of $\tau_g$.
\end{Prop}

\proof Part (c) clearly follows from (b).

We deduce (b) from Lemma \ref{XYLem}. 
Let us present the map $\KK^c_G\to HH_0(Bar^c(Sm))$ by an explicit cocycle.
Fix $f\in \KK\cap \H_K$. Consider the decomposition of $\mcS(G/U)^K$ as a direct
sum of spaces of functions on preimages of $K$ orbits in $G/B$. Then the action of $f$ 
on this space has zero block diagonal components with respect to this decomposition.
Also, the image of $f$ in $\Hh=\wh{\mcS(Y)}$
is contained in $\mcS(Y_0)$, while $\wh{\mcS(Y_0)}_G=0$ by the proof of Proposition \ref{SL2prop}.
Thus $f=\sum h_i-^{g_i}h_i$, $g_i\in G$, $h_i\in \wh{\mcS(Y_0)}$. 
It is easy to check that the image of $\bar{f}$ in  $HH_0(Bar^c(Sm))$
is represented by $(f,\sum g_i\delta_K\otimes h_ig^{-1})\in \H\oplus \Hh\otimes \Hh$.

It remains to check that for $h_i$ as above we have
\begin{equation}\label{Tr0}
Tr([g_i,\Pi]h_ig_i^{-1}\circ g, \wh{\mcS(X)})=0.
\end{equation}

We claim that the following stronger statement holds:
the operator $[g_i,\Pi]h_ig_i^{-1}$ acting on $\mcS(G/U)$ has zero diagonal
blocks with respect to the above block decomposition. To see this pick an open compact
subgroup $K'$ such that $h_i\in \H_{^{g_i^{-1}}K'}$ for all $i$. Then $g_i^{-1}$ sends
a summand of the decomposition corresponding to a $K'$-orbit  $U\subset G/B$ to a summand corresponding
to the $^{g^{-1}_i}K'$ orbits $g_i^{-1}(U)$, the endomorphism $h_i$ acts by an operator with zero
diagonal blocks with respect to the decomposition corresponding to $^{g^{-1}_i}K'$ orbits,
while $[g_i,\Pi]$ sends the summand corresponding to an orbit $g_i^{-1}(U)$ to the summand
corresponding to the orbit $U$ (this is clear since $\Pi$ preserves all the direct sum decompositions
above). Equality \eqref{Tr0} and hence part (b) of the Proposition follows.

The proof of part (a) occupies the next subsection.

\subsection{Tate cocyle and weighted orbital integral}
\subsubsection{The Tate cocycle}
A general reference for the following material is \cite{BFM} or \cite{BD},
especially 
\cite{BD}, \S 4.2.13, p 142; \S 7.13.18, p 344, developing the theme started in \cite{Tate}.

Let $V$ be a Tate vector space, then the Lie algebra $End(V)$
of continuous
endomorphism of $V$ has a canonical central extension, which we will
 denote by $\Endtil(V)$.

Let $End^b(V)\subset End(V)$  denote
the subspace of endomorphisms with bounded image, and $End^d(V)\subset End(V)$
be the subspace of endomorphisms having open kernel 
 (here b stands for bounded, and d for discrete).
Then $End^b(V)$, $End^d(V)$ are two-sided ideals in the associative algebra
$End(V)$, and hence ideals in the Lie algebra $End(V)$. 
Further, the
trace functional is defined on the intersection $End^{bd}=
End^b(V) \cap  End^d(V)$, $tr:End^{bd}\to k$; it satisfies
\begin{equation}\label{obnul}
\begin{array}{ll}
tr([E,E^{bd}])=0;\\
tr([E^b,E^d])=0
\end{array}
\end{equation} 
for $E^b\in End^b(V)$, $E^d\in End^d(V)$, 
$E^{bd}\in End^{bd}(V)$, $E\in End(V)$.

The central extension 
$\Endtil(V)\to End(V)$ is specified by the requirement that 
 it is trivialized on the ideals
$End^b(V)$, $End^d(V)$, i.e. the embeddings $End^b(V)\imbed End(V)$,
$End^d(V)\imbed End(V)$ are lifted to (fixed) homomorphisms 
 $s_b:End^b(V)\imbed \Endtil(V)$, $s_d: End^d(V)\imbed \Endtil(V)$,
where $s_b$, $s_d$ intertwine  the adjoint action of $End(V)$;
and for $E^{bd}\in End^{bd}(V)$ 
we have
\begin{equation}\label{Tate_can}
s_b(E^{bd})-s_d(E^{bd})=tr(E^{bd})\cdot c,
\end{equation}
where $c\in \Endtil(V)$ is the generator of the kernel of 
the projection $\Endtil(V)\to End(V)$.

Suppose now that a decomposition of $V$  
\begin{equation}\label{deco}
V=V^+\oplus V^- 
\end{equation} into a sum of a
bounded open 
 and a discrete subspaces is given,
 and let $\Pi$ denote the projection to the bounded open 
$V^+$ along the discrete $V^-$.
Then for any $E\in End(V)$ in the right-hand side
of $$E=E\cdot \Pi+E\cdot(Id-\Pi)$$
the first summand lies in $End^b(V)$, and the second one in $End^d(V)$.
Thus we get a splitting $s=s_\Pi:End(V)\to \Endtil(V)$,
\begin{equation} \label{split_idem}
s_\Pi(E)=s_b(E\cdot \Pi)+s_d(E\cdot(Id-\Pi)).
\end{equation}

It is also easy to see that for $E$ preserving $V_-$ (respectively, $V_-$)
the element $s_\Pi(E)$ is independent of the choice of the complement 
$V_-$ (respectively, $V_+$). Thus we get a canonical splitting $s_{V^+}$ (respectively, $s_{V^-}$) of the central extension
on the subalgebras $End_{V_+}(V)$, $End_{V^-}(V)$ consisting of endomorphisms
preserving $V_+$ (respectively, $V_-$).

Denote by $C(E_1,E_2)$ the corresponding 2-cocycle of $End(V)$,
i.e. a bi-linear functional, such that $$[s(E_1),s(E_2)]=
s([E_1,E_2])+C(E_1,E_2)\cdot c.$$
Then we have 
\begin{equation}\label{C}
C(E_1,E_2)=Tr(E_1\circ \Pi \circ E_2 \circ (Id-\Pi) -
E_2\circ \Pi \circ E_1 \circ (Id-\Pi)).
\end{equation}

Finally, suppose that another discrete cobounded space $W\subset V$ is fixed.
The splitting $s_W$ of the central extension  on $End_W(V)$ yields a linear functional
$\sigma_W=s_W-s_\Pi|_{End_W(V)}$ on $End_W(V)$.

\begin{example}
We have 
\begin{equation}\label{sigma_Id}
\sigma_W(Id)=\dim(V^+\cap W)-{\mathrm{co}}\dim_V(V^+ +W).
\end{equation}
Also, suppose that $F$ is an automorphism of $V$ such that either $F(V^+)\subset V^+$, or 
$V^+\subset F(V^+)$; set $d_{V^+}(F)=-\dim (V^+/F(V^+))$ in the former  and 

$d_{V_+}(F)=\dim (F(V^+)/V^+)$ in the latter case. Then 
\begin{equation}\label{C_F}
C(F,F^{-1})=d_{V^+}(F).
\end{equation}
Both equalities follow directly from the definitions.
\end{example}

Consider now the complex 
\begin{equation}\label{Bargen}
Cone\left( Bar(End_{V^+}(V))\oplus Bar(End_W(V)) \to Bar(End(V)) \right)[-1]
\end{equation}
and define a zero-cochain for the dual complex $\epsilon = (0,\sigma_W, C)$.

It is easy to check that the zero-cochain $\epsilon$ is in fact a cocycle
whose cohomology class does not depend on the choice of $V^-$ for 
a fixed $V^+$.

\subsubsection{Weighted orbital integral and Tate cocycle}
We now apply this in the following example. Consider the Tate space 
\begin{equation}\label{enterX}
V=\Hh\cong \wh{\mcS(Y)}\cong \wh{\mcS(X)}
\end{equation}
where $X$ was defined in the second paragraph of \S \ref{EHo},
 the first isomorphism was discussed above and the second one is induced by the inverse intertwining operator.
 Let $W$ be the image of $\H$ in the space \eqref{enterX} and $ \wh{\mcS(X)^+}$ be the space of functions supported on $X_+$ while $\wh{\mcS(X)^-}$ is the space of functions supported on the complement of $X_+$. 
 
The group $K_0$ acts on all these spaces. Fix a representation $\rho$ of $K_0$ and set
 $V_\rho=Hom_{K_0}(\rho, V)$,  $W_\rho=Hom_{K_0}(\rho, \H)$ and $V_\rho^\pm=Hom_{K_0}(\rho, \wh{\mcS(X)}^\pm)$.
 
 The complex \eqref{BarSmb} maps naturally to the complex \eqref{Bargen} constructed from $V=V_\rho$
 $V^\pm=V_\rho^\pm$ and $W=W_\rho$, let $a_\rho$ denote that map.
 
\begin{Lem}\label{WOeps} 
We have $a_\rho^*(\epsilon)=\int_{K_0} \tau_g Tr(g,\rho) dg$.
\end{Lem} 

\proof  
Equality of components in $(\Hh^{\otimes 2})^*$ follows by inspection.

It remains to check equality of components in $\H^*$.   Recall (see e.g. \cite[\S I.11]{ArIn})
that $WO_g(f)=\phi(0)$, where $\phi$ is a linear function on dominant weights
such that for large $\lambda$ we have
$$\phi(\lambda)=Tr(f\circ \Pi_\lambda \circ g,\H).$$ 
Here $f$ acts on $\H$ by convolution on the left, $g$ acts by right translation and
$\Pi_\lambda$ is the characteristic function of $G_{\leq \lambda}$ which is the
union of two sided $K_0=G(O)$ cosets corresponding to $\mu\leq \lambda$. 

It follows that for a locally constant function $\psi$ on $K_0$ supported on regular
semisimple elements there exists an affine linear function $\phi_\psi(\lambda)$, such that
$$\phi_\psi(\lambda)=Tr(f\circ \Pi_\lambda \circ \psi,\H) \ \ \ \ \ \ {\mathrm{for}} \ \ \lambda\gg 0,$$
$$\int_{K_0}WO_g(f)\psi(g)dg=\phi_\psi(0),$$
where $\psi$ acts by convolution on the right. 

Comparing this with the definition of Tate cocycle $\epsilon$ and of
the intertwining bimodule $\Hb'$ (see Example \ref{int_bim}) and using
Lemma \ref{XYLem}, we see that
for $\psi$ as above: 
$$\int_{K_0}WO_g(f)\psi(g)dg = \sigma_{\H} (f\otimes \psi),$$
where $f\otimes \psi$ acts on the Tate space \eqref{enterX} via its natural $\H$-bimodule
structure.
Both sides of the last equality are continuous in $\psi$ with respect to the $L^1$ norm: this is clear for the right hand side
and it follows from Theorem \ref{imbedding} for the left hand side. Thus validity of the equality
for $\psi$ supported on regular semisimple elements implies its validity for all $\psi$.    \qed

\subsubsection{Sheaves of algebras on curves}
We now apply the above construction in the following setting. Let $C$ be a smooth\footnote{This assumption is likely unnecessary but it allows one to 
simplify the statements and the proofs.} complete curve with a finite collection
of points ${\bf{x}}=\{x_1,\dots,x_n\}$. Let $\hat{C}_{\bf{x}}$, $\hat{C}_{\bf{x}}^0$
be the formal neighborhood and the punctured formal neighborhood of $\bf{x}$ respectively. Let
$\Ve$ be a torsion free coherent sheaf on $C$, and set $V=\Gamma(\hat{C}_{\bf{x}}^0, \Ve)$,
 $W=\Gamma(C\setminus \bf{x},\Ve)$, $V^+=\Gamma(\hat{C}_{\bf{x}}, \Ve)$.
Fixing formal coordinates $z_i$ at $x_i$ and trivializations of $\Ve$ on the formal neighborhood of $x_i$ we get
an isomorphism $V\cong \Ce((z))^N$ sending $V_+$ to $\Ce[[z]]^N$, thus we get a splitting $V=V^+\oplus V^-$
where $V^-\cong z^{-1}\Ce[z^{-1}]^N$. 

Let $\A$ be a torsion free 
coherent sheaf of algebras on $C$ with a right action on $V$;
we assume also that $\A$ has finite homological dimension (i.e.
that the algebra of sections of $\A$ on an affine open set has this property).

Let $A=\Gamma(C\setminus \bf{x},\A)$, $\hat{A}^+=\Gamma(\hat{C}_{\bf{x}}, \A)$
and  $\hat{A}=\Gamma(\hat{C}_{\bf{x}}^0, \A)$.
 Then the complex 
 \begin{equation}\label{BarA}
Cone\left( Bar(A)\oplus Bar(\hat{A}^+)\to Bar(\hat{A}) \right)[-1]
\end{equation}
computes Hochschild homology of the category $\A-mod$. On the other hand, it maps naturally 
to the complex \eqref{Bargen}, let $\alpha$ denote this map.

\begin{lem}
Suppose that $\Ve$ is a locally projective sheaf of right $\A$-modules.
Then for a coherent sheaf
$\MM$ of (left) $\A$-modules we have:
\begin{equation}\label{VeM}
\langle ch(\MM), \alpha^*(\epsilon)\rangle = \chi(\Ve\otimes_{\A}\MM),
\end{equation}
where $\chi$ denotes Euler characteristic.
\end{lem}

\proof Both sides of \eqref{VeM} do not change if we replace $\A$ by 
$\underline{End}_{\O_C}(\Ve)^{op}$ and $\MM$ by $\underline{End}_{\O_C}(\Ve)^{op}\otimes_\A \MM$.
Thus we can assume without loss of generality that $\A=\underline{End}_{\O_C}(\Ve)^{op}$.
In that case we have a canonical equivalence of categories $Coh(C)\cong \A-mod$,
$\F\mapsto \F\otimes_{\O_C}\Ve^*$.

Furthermore, the operator $[\F]\mapsto [\F\otimes \Ve]$ 
induces an automorphism of the rational Grothendieck group $K^0(Coh(C))\otimes \Qu$,
thus it suffices to prove \eqref{VeM} for $\A=\underline{End}_{\O_C}(\Ve)^{op}$, $\MM=\F\otimes \Ve^*$, where $\F=\Ve\otimes \F'$,
thus $\F=\F'\otimes _{\O_C}\A$.
Now, both sides of \eqref{VeM} do not change if we replace $\A$ by $\O_C$
and $\MM$ by $\F'$; the locally projective module $\Ve$ is replaced by 
the corresponding coherent sheaf $\Ve|_{\O_C}$. Thus we have reduced to the case $\A=\O_C$.

It is easy to see that both sides of \eqref{VeM} are additive on short exact sequences
as a function of both $\Ve$ and $\F$.
A locally free sheaf on a curve admits a filtration whose associated graded is a sum of line bundles;
also, since $C$ is smooth the Grothendieck group $K^0(Coh(C))$ is generated by line bundles.
Thus  we can assume without loss of generality  that both $\MM$ and $\Ve$ are line bundles. 

If $\MM$ is a trivial line bundle then $ch(\MM)$ is represented by the cocycle $(1,1,1\otimes 1)$, so \eqref{VeM} follows from \eqref{sigma_Id}. 
Notice also that both sides of \eqref{VeM} for a fixed $\Ve$ depend only on the degree of the line
bundle $\MM$ (here we assume without loss of generality that the curve $C$ is irreducible): this is clear for the right hand side, while 
for the left hand side it follows from the fact that a regular function on an abelian variety is constant.
Thus it suffices to consider the case when $\MM=\O(nx)$, $x\in \bf{x}$. In this case $ch(\MM)$ is represented by the cocycle
$(1,1,f\otimes f^{-1})$ where $f$
is a function on the formal punctured neighborhood of $\bf{x}$ having order $n$ at $x$ and order $0$ at other points.
In this case \eqref{VeM} follows from \eqref{sigma_Id} and \eqref{C_F}. \qed

\begin{cor} \label{Eulerchar} a)
Suppose that $\Ve\cong {\bf e} \A $ is a direct summand in a free rank one module for an idempotent
${\bf e} \in \Gamma(C,\A)$. Then for a coherent sheaf
$\MM$ of $\A$-modules we have:
$$\langle ch(\MM), \alpha^*(\epsilon)\rangle = \chi({\bf e}\MM),$$
where $\chi$ denotes the Euler characteristic.

b) Suppose that $\MM$ has finite support contained in $C\setminus \bf{x}$. Then
the equality in (a) holds under a weaker assumption that $\Ve|_{C\setminus \bf{x}}\cong \A {\bf e} |_{C\setminus \bf{x}}$.
\end{cor}

\subsubsection{Proof of Proposition \ref{tau_prop}(a)}
It suffices to check that we get an equality upon averaging against a character
of any representation $\psi$ of $K_0=G(O)$. 

In view of Corollary \ref{Eulerchar}(b) it suffices to check that 
the cochain $\int_{K_0} \tau_g Tr(g,\psi) dg$  is obtained as in Corollary \ref{Eulerchar}
for an appropriate choice of trivialization for $\Ve$ on the formal neighborhood
of $\bf{x}$, where 
 $C\subset \Zb$ is the union of one dimensional components containing representations with nonzero $K$-invariant vectors, 
 $\bf{x}=(\Zb\setminus \Z)\cap C$, $\A=\Hb$, $z_i$ are the natural
local coordinates and $\Ve=Hom_{K_0}(\psi,\Hb')$. Applying Lemma \ref{WOeps}
we reduce to showing that $a_\psi^*(\epsilon)$ has the required form.

It follows from Lemma \ref{XYLem} that the space $V_+$ of sections of $\Hb'$ on the formal
neighborhood of $\Zb\setminus \Z$ is identified with $\wh{\mcS(X_+)}^{K\times K}$ (notation
introduced in the second paragraph of \S \eqref{EHo});
moreover, we can choose a trivialization of the sheaf $\Hb$ on the formal neighborhood
of $\Zb\setminus \Z$ so that constant sections correspond to functions supported on
the set $X(O)=X_+ \setminus (tX_+)$.
Then the space $V_-$ is identified with 
$\mcS(X\setminus X_+)^{K\times K}$. The desired equality now follows by inspecting
the definitions.  \qed

\subsection{Proof of part (b) of Theorem \ref{char_form}}
For $g\in K_0$ the formula follows from Proposition \ref{tau_prop}. The general case 
of a compact element follows by conjugating with an element of $GL(2,F)$. 
\qed




  \end{document}